\documentclass[11pt,psfig,a4]{article}

\usepackage{geometry}
\usepackage{latexsym}
\usepackage{graphics}
\usepackage{amsmath,amsthm,amsfonts,amssymb}
\usepackage{hyperref}

\theoremstyle{plain}
\def\beq{\begin{eqnarray}}
\def\eeq{\end{eqnarray}}
\def\beqq{\begin{eqnarray*}}
\def\eeqq{\end{eqnarray*}}

\def\qed{\hfill$\sqcap\kern-8.0pt\hbox{$\sqcup$}$\\}
\def\beq{\begin{eqnarray}}
\def\eeq{\end{eqnarray}}
\def\be{\begin{equation}}
\def\ee{\end{equation}}

\numberwithin{equation}{section}

\newtheorem{theorem}{Theorem}[section]

\newtheorem{claim}[theorem]{Claim}

\newtheorem{corollary}[theorem]{Corollary}

\newtheorem{lemma}[theorem]{Lemma}
\newtheorem{proposition}[theorem]{Proposition}

\theoremstyle{definition}

\theoremstyle{remark}
\newtheorem{remark}[theorem]{Remark}

\makeatother

\begin{document}

\title{Construction of conformally compact Einstein manifolds}

\author{Dezhong Chen
\thanks{Department of Mathematics and Statistics, McMaster University, Hamilton, Ontario, L8S 4K1, Canada.
E-mail address: chend6@math.mcmaster.ca}}

\maketitle

\begin{abstract}
We produce some explicit examples of conformally compact Einstein manifolds, whose conformal compactifications are foliated by Riemannian products of a closed Einstein manifold with the total space of a principal circle bundle over products of K\"{a}hler-Einstein manifolds. We compute the associated conformal invariants, i.e., the renormalized volume in even dimensions and the conformal anomaly in odd dimensions. As a by-product, we obtain some Riemannian products with vanishing $Q$-curvature.
\end{abstract}

\section{Introduction}\label{intr}

The idea of conformally compactifying Einstein manifolds appeared first in the work of Penrose \cite{Pen65}. More than three decades later, physicists rediscovered its role in Maldacena's AdS/CFT correspondence \cite{Mal98, GubKlePol98, Wit98}, which has become a very active area in string theory. In the mathematical community, the study of conformally compact Einstein manifolds was initiated by Fefferman and Graham \cite{FefGra85} in their search for new conformal invariants. Since then, a large amount of work has been devoted to a general theory on the existence of conformally compact Einstein manifolds \cite{GraLee91, Biq00, Lee06, And08}, and their applications in conformal geometry \cite{GraZwo03, FefGra02, ChaQinYan04}.

A conformally compact Einstein manifold is a complete Einstein manifold with negative scalar curvature, which can be conformally deformed and extended to a compact Riemannian manifold with boundary. The simplest example of a conformally compact Einstein manifold is hyperbolic space on which the hyperbolic metric is conformally related to the restriction of the Euclidean metric to the open unit ball, with the closed unit ball as its conformal compactification. The level sets of the radius function on the closed unit ball are round spheres, which shrink smoothly towards the origin. For hyperbolic space of even dimensions, the odd-dimensional foliating spheres may be viewed as the total spaces of Hopf fibrations. In this paper, we will replace the odd-dimensional spheres by the total spaces of appropriate principal circle bundles over products of K\"{a}hler-Einstein manifolds to obtain new conformally compact Einstein manifolds.

\begin{theorem}\label{positive}
Let $N^n$, $n\ge0$, be a closed Einstein manifold with positive scalar curvature. Let $V_1=\mathbb{C}P^{n_1}$, $n_1\ge0$, with normalized Fubini-Study metric $h_1$, and $\{(V_i^{n_i},h_i)\}_{2\le i\le r}$ be Fano K\"{a}hler-Einstein manifolds respectively with first Chern class $p_ia_i$, where $p_i\in\mathbb{Z}_+$ is the Einstein constant of $h_i$, and $a_i\in H^2(V_i;\mathbb{Z})$ is an indivisible class. For $q=(q_1,\cdots,q_r)$, $q_i\in\mathbb{Z}\backslash\{0\}$, let $P_q$ be the principal circle bundle over $V_1\times\cdots\times V_r$ with Euler class $\oplus_{i=1}^rq_ia_i$, where we assume that $|q_1|=1$. Then there exist at least a one-parameter family of non-homothetic conformally compact Einstein metrics on the product of $N$ with the total space of the complex $(n_1+1)$-disc bundle over $V_2\times\cdots\times V_r$ obtained from $P_q$ by blowing down the Hopf fibration $S^{2n_1+1}\rightarrow\mathbb{C}P^{n_1}$.
\end{theorem}

\begin{remark}
When $N$ is only a point, the Einstein condition should be understood as the condition that the constant $\epsilon$ in Eq. \eqref{ce} is positive, zero or negative (see Proposition \ref{eif} below for clarity).
\end{remark}

\begin{remark}
When $n_1=0$, $V_1=\mathbb{C}P^0$ is a point, and $p_1=1$. In this case, we just collapse the fibre circle of $P_q$ to obtain a complex one-disc bundle over $V_2\times\cdots\times V_r$.
\end{remark}

\begin{theorem}\label{zero}
Let $N^n$, $n\ge0$, be a closed Einstein manifold with zero scalar curvature. Let $V_1=\mathbb{C}P^{n_1}$, $n_1\ge0$, with normalized Fubini-Study metric $h_1$, and $\{(V_i^{n_i},h_i)\}_{2\le i\le r}$ be Fano K\"{a}hler-Einstein manifolds respectively with first Chern class $p_ia_i$, where $p_i\in\mathbb{Z}_+$ is the Einstein constant of $h_i$, and $a_i\in H^2(V_i;\mathbb{Z})$ is an indivisible class. For $q=(q_1,\cdots,q_r)$, $q_i\in\mathbb{Z}\backslash\{0\}$, let $P_q$ be the principal circle bundle over $V_1\times\cdots\times V_r$ with Euler class $\oplus_{i=1}^rq_ia_i$, where we assume that $|q_1|=1$. If in addition $(n_1+1)|q_i|>p_i$, $2\le i\le r$, then there exist a one-parameter family of non-homothetic conformally compact Einstein metrics on the product of $N$ with the total space of the complex $(n_1+1)$-disc bundle over $V_2\times\cdots\times V_r$ obtained from $P_q$ by blowing down the Hopf fibration $S^{2n_1+1}\rightarrow\mathbb{C}P^{n_1}$.
\end{theorem}

In the above two theorems, we require that all of the base factors $(V_i,h_i)$ have positive Einstein constants. However, in case some of $(V_i,h_i)$, $2\le i\le r$, have non-positive Einstein constants, we can still construct conformally compact Einstein metrics in the same manner provided $N$ has negative scalar curvature.

\begin{theorem}\label{np}
Let $N^n$, $n\ge0$, be a closed Einstein manifold with negative scalar curvature. Let $V_1=\mathbb{C}P^{n_1}$, $n_1\ge0$, with normalized Fubini-Study metric $h_1$, and $\{(V_i^{n_i},h_i)\}_{2\le i\le r}$ be closed K\"{a}hler-Einstein manifolds respectively with first Chern class $p_ia_i$, where $p_i\in\mathbb{Z}$ is the Einstein constant of $h_i$, and $a_i\in H^2(V_i;\mathbb{Z})$ is an indivisible class. In the case of $p_i=0$, we assume the K\"{a}hler class of $h_i$ is $2\pi a_i$, i.e., $V_i$ is a Hodge manifold. For $q=(q_1,\cdots,q_r)$, $q_i\in\mathbb{Z}\backslash\{0\}$, let $P_q$ be the principal circle bundle over $V_1\times\cdots\times V_r$ with Euler class $\oplus_{i=1}^rq_ia_i$, where we assume that $|q_1|=1$. If in addition $(n_1+1)|q_i|>p_i$, $2\le i\le r$, then there exist a one-parameter family of non-homothetic conformally compact Einstein metrics on the product of $N$ with the total space of the complex $(n_1+1)$-disc bundle over $V_2\times\cdots\times V_r$ obtained from $P_q$ by blowing down the Hopf fibration $S^{2n_1+1}\rightarrow\mathbb{C}P^{n_1}$.
\end{theorem}

In Theorem \ref{np} the additional condition $(n_1+1)|q_i|>p_i$ always holds true for $p_i\le0$ as $q_i\ne0$. We therefore have the following immediate consequence of Theorem \ref{np}.
\begin{corollary}
Let $N^n$, $n\ge0$, be a closed Einstein manifold with negative scalar curvature. Let $V_1=\mathbb{C}P^{n_1}$, $n_1\ge0$, with normalized Fubini-Study metric $h_1$, and $\{(V_i^{n_i},h_i)\}_{2\le i\le r}$ be closed non-positive K\"{a}hler-Einstein manifolds respectively with first Chern class $p_ia_i$, where $p_i\in\mathbb{Z}_-\cup\{0\}$ is the Einstein constant of $h_i$, and $a_i\in H^2(V_i;\mathbb{Z})$ is an indivisible class. In the case of $p_i=0$, we assume the K\"{a}hler class of $h_i$ is $2\pi a_i$. For $q=(q_1,\cdots,q_r)$, $q_i\in\mathbb{Z}\backslash\{0\}$, let $P_q$ be the principal circle bundle over $V_1\times\cdots\times V_r$ with Euler class $\oplus_{i=1}^rq_ia_i$, where we assume that $|q_1|=1$. Then there exist a one-parameter family of non-homothetic conformally compact Einstein metrics on the product of $N$ with the total space of the complex $(n_1+1)$-disc bundle over $V_2\times\cdots\times V_r$ obtained from $P_q$ by blowing down the Hopf fibration $S^{2n_1+1}\rightarrow\mathbb{C}P^{n_1}$.
\end{corollary}

\begin{remark}
When $n=0$ and $N$ is a point, a couple of cases of Theorems \ref{positive}, \ref{zero} and \ref{np} have been contained implicitly in earlier work. For example, the case $n_1>0$, $n_2=0$ and $r=2$, where the underlying manifold is the Euclidean ball $B^{2n_1+2}$, was obtained in \cite{Ber82,Ped86}, as was the case $n_1=0$, $n_2>0$ and $r=2$ (cf. also \cite{Cal79}). A generalization of the latter case was given by \cite[Theorem 1.6(d)]{WanWan98}.
\end{remark}

\begin{remark}
The lowest-dimensional examples which Theorems \ref{positive}, \ref{zero} and \ref{np} yield are $B^4$ and some nontrivial complex one-disc bundles over Riemann surfaces.
\end{remark}

From our point of view, a conformally compact Einstein manifold and all of its conformal compactifications are two sides of one object, connected by the collection of defining functions. If we set as a goal the existence of a conformally compact Einstein metric on a given manifold, then it seems more flexible to seek one of its conformal compactifications and the corresponding defining function because they are a priori more plentiful. Our strategy for proving Theorems \ref{positive}, \ref{zero} and \ref{np} is thus to construct a compact Riemannian manifold with boundary, which serves as a conformal compactification, and a suitable defining function on it. This leads us to consider a special type of conformal deformation as follows (compare to \cite[Theorem 1(I)]{Cle08}).

Let $(M^m,g_M)$ and $(N^n,g_N)$, $p=m+n>3$, be two connected Riemannian manifolds. In particular, assume $(N,g_N)$ is a closed Einstein manifold with Einstein constant $\epsilon$. Let $\rho:M\rightarrow\mathbb{R}_+$ be a positive smooth function. Then the conformal metric $g=\rho^{-2}(g_M+g_N)$ is an Einstein metric on the product $M\times N$ if $\rho$ satisfies a quasi-Einstein-like equation (see \eqref{ce} below), which is also necessary for $g$ being an Einstein metric when $n>0$. In case $M$ is the total space of a disc bundle as in Theorems \ref{positive}, \ref{zero} and \ref{np}, with both $g_M$ and $\rho$ on the hypersurfaces $P_q$ depending on the radius parameter, Eq. \eqref{ce} reduces to a system of ordinary differential equations, which happens to admit nontrivial exact solutions. Then a compact Riemannian manifold-with-boundary $(\bar{M}\times N,g_{\bar{M}}+g_N)$ and a negative Einstein manifold $(M\times N,g)$ can be built out of these data if we choose the involved parameters carefully. It remains to check that $(M\times N,g)$ is indeed complete.
\begin{remark}
When $n>0$, Eq. \eqref{ce} is intimately connected to the so-called quasi-Einstein equation (see \eqref{qee} below), which arises from constructions of warped product Einstein manifolds \cite[Corollary 9.107]{Bes87}. It has been shown by Kim and Kim \cite[Theorem 1]{KimKim03} that a non-positive warped product Einstein manifold with closed base must be a Riemannian product. In the language of warped product structures, our negative Einstein manifold $(M\times N, g)$ has the complete noncompact manifold $(M,\rho^{-2}g_M)$ as base manifold. Therefore Kim and Kim's theorem does not apply to the present situation.
\end{remark}

We turn now to consider the conformal invariants associated to a conformally compact Einstein manifold $(X^p,g)$. Since $(X,g)$ always has infinite volume, it is therefore natural to investigate the asymptotic behavior of the volume function. It is well-known (cf. \cite{Gra00}) that certain conformal invariants associated to $(X,g)$ appear in the asymptotic expansion of the volume function (see \eqref{pe} and \eqref{po} below). Basically, there are two cases distinguished by the parity of the dimension $p$.

When $p$ is even, the constant term in the asymptotic expansion \eqref{pe}, called the renormalized volume, is an invariant of $(X,g)$ itself. The differential one-form of the renormalized volume is proportional to the Brown-York quasi-local stress-energy tensor of the usual Einstein-Hilbert action \cite[(1.11)]{And05}, and via the AdS/CFT correspondence, it corresponds to the expectation value of the stress-energy tensor of the dual CFT on the conformal infinity $(\partial X, [\rho^2g|_{\partial X}])$, where $\rho$ is a defining function for boundary $\partial X$ (see \S\ref{pre} for definition). From a mathematical point of view, the importance of the renormalized volume stems from its role in the Gauss-Bonnet-Chern formula (see \eqref{gbc} below). This fact has brought about some progress in understanding the topology of four-manifolds (cf. \cite{ChaQinYan04}). We will derive an explicit formula for the renormalized volumes of the even-dimensional conformally compact Einstein manifolds constructed in Theorems \ref{positive}, \ref{zero} and \ref{np} (cf. Theorem \ref{rvv}).

When $p$ is odd, the coefficient of the logarithmic term in the asymptotic expansion \eqref{po}, called the conformal anomaly, is a conformal invariant of $(\partial X, [\rho^2g|_{\partial X}])$, which agrees with the integral of Branson's $Q$-curvature \cite{GraZwo03, FefGra02} up to a dimensional constant. Furthermore, the $Q$-curvature can be read off from the asymptotic behavior of some particular solution $\tilde{U}$ of the Poisson equation $\triangle_g\tilde{U}=p-1$ on $(X,g)$ (cf. \cite[Theorem 4.1]{FefGra02}), where $\triangle_g$ is the positive Laplace operator. It is not hard to see that all of the odd-dimensional conformally compact Einstein manifolds constructed in Theorems \ref{positive}, \ref{zero} and \ref{np} have zero conformal anomaly. A little more work following the original proof of \cite[Theorem 3.1]{FefGra02} shows that
\begin{theorem}\label{vca}
There exists a representative metric with zero $Q$-curvature in the conformal infinity of each odd-dimensional conformally compact Einstein manifold constructed in Theorems \ref{positive}, \ref{zero} and \ref{np}.
\end{theorem}

Theorem \ref{vca} together with Theorems \ref{positive}, \ref{zero} and \ref{np} gives rise to a large class of Riemannian products with vanishing $Q$-curvature. For example, if we take $n_1>0$, $n_2=0$ and $r=2$ in Theorems \ref{positive}, \ref{zero} and \ref{np}, then Theorem \ref{vca} yields
\begin{corollary}\label{rpzq}
Let $(N^n,g_N)$ be an odd-dimensional closed Einstein manifold. For each $k>0$, there exist a one-parameter family of squashed spheres $S^{2k+1}$ such that the Riemannian products $S^{2k+1}\times(N^n,g_N)$ have zero $Q$-curvature. In particular, if $(N^n,g_N)$ has non-positive Einstein constant $-(n-1)$, then the standard sphere $S^{2k+1}(1)$ of radius $1$ lies in this one-parameter family of squashed spheres.
\end{corollary}

The latter part of Corollary \ref{rpzq} has the following generalization due to Graham \cite{gra09}.
\begin{proposition}\label{graham}
Let $(M^m,g_M)$ and $(N^n,g_N)$ be two odd-dimensional closed Einstein manifolds respectively with Einstein constants $m-1$ and $-(n-1)$. Then the Riemannian product $(M^m\times N^n,g_M+g_N)$ has zero $Q$-curvature.
\end{proposition}
\begin{remark}
We would like to thank Juhl for bringing to our attention \cite{Juh09}, where he also obtained Proposition \ref{graham} (see Corollary 10.3 therein). In fact, Juhl has a general formula for all $Q$-curvatures of Riemannian products as in Proposition \ref{graham} (Corollary 10.2).
\end{remark}

One simple way to see this is suggested to the author by Graham as follows. Notice that the $Q$-curvature can be expressed in terms of Ricci curvature and its covariant derivatives only (cf. \cite[Proposition 3.5 and the remark following it]{FefGra07}). Therefore an even-dimensional Riemannian manifold has constant $Q$-curvature provided its Ricci tensor is parallel. In this situation, the Riemannian manifold is locally a product of Einstein manifolds, and the $Q$-curvature is given by a universal formula involving the Einstein constants only. We turn back to Proposition \ref{graham}. It is clear that $(M\times N,g_M+g_N)$ has parallel Ricci tensor. Hence it has constant $Q$-curvature. In order to evaluate the constant, one may, without loss of generality, assume that $(M,g_M)$ and $(N,g_N)$ are respectively of constant curvature $\mbox{sgn}(m-1)$ and $\mbox{sgn}(-(n-1))$, where $\mbox{sgn}$ is the usual sign function. Thus $(M\times N,g_M+g_N)$ is conformally flat with zero Euler characteristic. Proposition \ref{graham} now follows from the well-known fact that on closed conformally flat manifolds the integral of $Q$-curvature is a constant multiple of the Euler characteristic.

\begin{remark}
The representative metrics mentioned in Theorem \ref{vca}, denoted by $g_b$, are given by \eqref{can} below. It is easy to see that $g_b$ has parallel Ricci tensor iff its restriction to the principal circle bundle $P_q$ is an Einstein metric (cf. \cite[9.25]{Bes87}, Eq. \eqref{at} and Lemma \ref{nv}), which, however, generally does not hold. Therefore one cannot simply apply the above arguments to show that $g_b$ has constant, and hence vanishing $Q$-curvature.
\end{remark}

Many authors have addressed the existence of a representative metric with constant $Q$-curvature in a given conformal class of Riemannian metrics on an even-dimensional closed manifold. Such a metric necessarily minimizes the $L^2$-norm of $Q$-curvature among conformal metrics with fixed volume. Several existence theorems for constant $Q$-curvature metrics under some generic assumptions can be found in \cite{ChaYan95, Bre03, Ndi07, DjaMal08}. For example, it is shown by Ndiaye \cite[Theorem 1.1]{Ndi07} (cf. \cite[Theorem 1.1]{DjaMal08} for $4$-dimensional case) that a conformal class contains a representative metric with constant $Q$-curvature provided that the kernel of the critical GJMS operator $P_0$ consists of constant maps, and the integral of $Q$-curvature does not equal any positive integer multiple of the counterpart on standard sphere. In Theorem \ref{vca}, the integral of $Q$-curvature is zero, and hence the latter condition holds true. But it is not clear to the author whether the kernel of the corresponding operator $P_0$ consists of constant maps, except for the cases $S^1\times S^3(1)$ and its finite quotients where $\ker P_0=\mathbb{R}$ by \cite[Theorem A]{Gur99}. Therefore we do not know in general whether our vanishing $Q$-curvature metrics are predicted by Ndiaye's theorem.
\begin{remark}
If $\ker P_0\supsetneqq\mathbb{R}$, then Ndiaye's theorem does not apply to this case. On the other hand, if $\ker P_0=\mathbb{R}$, then Ndiaye's theorem asserts our vanishing $Q$-curvature metrics, which turn out to be unique up to homothety in the given conformal classes. However, our results are reasonably explicit.
\end{remark}

The organization of this paper is as follows. In \S\ref{pre}, we review some of the fundamental material about conformally compact Einstein manifolds. In \S\ref{qele}, we introduce the aforementioned quasi-Einstein-like equation \eqref{ce}, and discuss the relevant regularity issues. In \S\ref{constr}, we prove Theorems \ref{positive}, \ref{zero} and \ref{np} through a detailed account of the constructions of conformally compact Einstein metrics, and then describe some simple but important examples. In \S\ref{ci}, we compute the conformal invariants associated to the conformally compact Einstein manifolds constructed in \S\ref{constr}, and prove Theorem \ref{vca} and Corollary \ref{rpzq}.

\textbf{Acknowledgment.} The author is grateful to his PhD supervisor, Prof. M. Y. Wang, for his consistent help and inspiration underlying all of this work. He wishes to thank Drs. J. Cuadros, A. S. Dancer, C. R. Graham, P. Guan, A. Juhl, M. Min-Oo and A. J. Nicas for their interests in this paper and many very valuable comments.

\textbf{Note added in proof.} We would like to thank the referee for bringing to our attention \cite{ManSte06}, where the physicists constructed very similar conformally compact Einstein manifolds. However, our construction is interesting in its own right since our points of view are quite different from theirs. In particular, instead of seeking solutions to the Einstein equation directly, we are looking for solutions to the quasi-Einstein-like equation \eqref{ce}.

\section{Preliminaries}\label{pre}

In this section, we will briefly review some elementary properties of conformally compact Einstein manifolds. For more details, one may refer to \cite{Gra00, And05}.

Let $\bar{X}^{p}=X\cup\partial X$, $p>2$, be a $p$-dimensional smooth compact manifold with interior $X$ (the bulk) and boundary $\partial X$. A defining function $\rho$ for $\partial X$ is a nonnegative smooth function on $\bar{X}$ such that $\rho^{-1}(\{0\})=\partial X$ and $d\rho\ne0$ on $\partial X$. A complete Einstein metric $g$ on $X$ is said to be conformally compact if there exists a defining function $\rho$ such that the conformal metric $\bar{g}=\rho^2g$ extends to a smooth metric on $\bar{X}$. In this case, we call $(X,g)$ a conformally compact Einstein manifold, and $(\bar{X},\bar{g})$ its conformal compactification via $\rho$. A straightforward computation then shows that the sectional curvatures of $g$ are asymptotic to $-|d\rho|_{\bar{g}}^2$ near $\partial X$, where the Ricci tensors of $g$ are thus asymptotic to $-(p-1)|d\rho|_{\bar{g}}^2g$. Since $g$ is an Einstein metric and $d\rho$ is transverse to $\partial X$, $|d\rho|_{\bar{g}}^2$ must be a positive constant on $\partial X$. Hence $g$ has a negative Einstein constant. From now on, we shall normalize $g$ such that $\mbox{Ric}(g)=-(p-1)g$, which forces $|d\rho|_{\bar{g}}^2$ to be identically $1$ on $\partial X$. It follows that the sectional curvatures of $g$ approach $-1$ near $\partial X$. Therefore the normalized conformally compact Einstein metrics are asymptotically hyperbolic.

It is clear that if $\rho$ is a defining function for $\partial X$, then so is $\hat{\rho}\rho$ for any positive smooth function $\hat{\rho}$ on $\bar{X}$. Hence the restriction of $\bar{g}$ to $\partial X$ gives rise to a conformal class of metrics on $\partial X$, called the conformal infinity of $(X,g)$. By \cite[Lemma 2.1]{Gra00}, for any metric $\gamma$ in the conformal infinity, there exists a unique defining function $\sigma$ for $\partial X$, called geodesic defining function, such that on the induced collar neighborhood of $\partial X$ in $\bar{X}$, identified with $[0,\delta)\times\partial X$ for some $0<\delta\ll1$,
\begin{align}\label{gc}
g=\sigma^{-2}(d\sigma^2+g_\sigma),
\end{align}
where $g_\sigma$, $\sigma\in[0,\delta)$, are a one-parameter family of metrics on $\partial X$ with $g_0=\gamma$. It follows that the volume of the set $\{\sigma>\delta\}$ with respect to $g$ will diverge to infinity as $\delta$ decreases to $0$. A careful examination shows that the volume function actually has some very interesting asymptotic behavior (cf. \cite[(3.3) and (3.4)]{Gra00}).
\begin{itemize}
  \item When $p$ is even,
  \begin{align}\label{pe}
  \mbox{Vol}_g(\{\sigma>\delta\})=c_0\delta^{1-p}+(\mbox{odd powers})+c_{p-2}\delta^{-1}+V+o(1);
  \end{align}
  \item When $p$ is odd,
  \begin{align}\label{po}
  \mbox{Vol}_g(\{\sigma>\delta\})=c_0\delta^{1-p}+(\mbox{even powers})+c_{p-3}\delta^{-2}-L\log\delta+V+o(1),
  \end{align}
\end{itemize}
where the constants $c_i$ are explicitly computable from the geometry of $(\partial X,\gamma)$. For example, $c_0=(p-1)^{-1}\mbox{Vol}_\gamma(\partial X)$.

Of particular interest to us in these asymptotic expansions are the constant term $V$ and the coefficient $L$ of the logarithmic term, called respectively the renormalized volume and the conformal anomaly. It follows from \eqref{pe} that $L$ is identically zero in even dimensions. Otherwise, the renormalized volume and the conformal anomaly may a priori depend on the choice of representative metric $\gamma$ in the conformal infinity. However, the next two remarks indicate that the renormalized volume in even dimensions and the conformal anomaly in odd dimensions are actually independent of the choice of $\gamma$.

When $p$ is even, the renormalized volume $V$ appears in the following version of the Gauss-Bonnet-Chern formula \cite[Theorem 4.6]{ChaQinYan08}
\begin{align}\label{gbc}
\int_X\mathfrak{W}_gdv_g+(-1)^{p/2}\pi^{-(p+1)/2}\Gamma(\frac{p+1}{2})V=\chi(X),
\end{align}
where $\Gamma$ is the usual gamma function, $\chi(X)$ is the Euler characteristic of $X$, and $\mathfrak{W}_g$ is a finite sum of contractions of the Weyl tensor $W_g$ and its covariant derivatives. For example, when $p=4$, $\mathfrak{W}_g=(8\pi^2)^{-1}|W_g|_g^2$, and \eqref{gbc} reads (cf. \cite{And01})
\begin{align*}
\int_X|W_g|_g^2dv_g+6V=8\pi^2\chi(X).
\end{align*}
It follows that $V\le\frac{4}{3}\pi^2\chi(X)$, with equality iff $(X,g)$ has constant sectional curvature $-1$. This is precisely
\begin{proposition}\label{4h}
The $4$-dimensional hyperbolic metric achieves the maximal renormalized volume $\frac{4}{3}\pi^2$ among all of conformally compact Einstein metrics on the Euclidean ball $B^4$.
\end{proposition}

When $p$ is odd, the conformal anomaly $L$ is related to Branson's $Q$-curvature \cite{Bra95} of the conformal infinity $(\partial X,[\gamma])$ by the formula (cf. \cite{GraZwo03, FefGra02})
\begin{align}\label{confa}
L=\frac{(-1)^{(p-1)/2}2^{2-p}}{((p-1)/2)!((p-3)/2)!}\int_{\partial X}Q_\gamma dv_\gamma.
\end{align}

Recall that the $Q$-curvature is defined originally for even-dimensional closed manifolds (cf. \cite{Bra95}), whose integral is an invariant of the conformal class. Under a conformal deformation of metric $\hat{g}=e^{2f}g$, the $Q$-curvature changes according to the law
\begin{align*}
\hat{Q}=e^{-2kf}(Q+P_0f),
\end{align*}
where $2k$ is the dimension of the underlying manifold, and $P_0$ is the formally self-adjoint $2k$-th order GJMS operator \cite{GraJenMasSpa92} with respect to $g$, which equals respectively the Laplace operator in dimension $2$ and the Paneitz operator in dimension $4$. Since the general definition of $Q$-curvature is quite involved, explicit formulas for $Q$-curvature are only available in low dimensions (cf. \cite{GovPet03, GraJuh07}). For example, in dimension $4$, we have
\begin{align}\label{qcur}
Q=\frac{1}{6}(\triangle R+R^2-3|\mbox{Ric}|^2),
\end{align}
where $R$ denotes the scalar curvature.

\section{A quasi-Einstein-like equation}\label{qele}

In this section, we study Einstein manifolds from the point of view of conformal geometry. Such a standpoint is quite natural in the spirit of conformally compact Einstein manifolds.

\subsection{A simple observation}\label{ob}

Let $(M^m,g_M)$ and $(N^n,g_N)$, $p=m+n>2$, be two connected Riemannian manifolds, and $\rho$ be a positive smooth function on $M$. Consider now the conformal metric $g=\rho^{-2}(g_M+g_N)$ on the product $M\times N$, where $g_M+g_N$ is the usual product metric. The Ricci tensor of $g$ is given by \cite[Theorem 1.159]{Bes87}
  \begin{eqnarray*}
  \mbox{Ric}(g)&=&\mbox{Ric}(g_M)+\mbox{Ric}(g_N)+(p-2)\frac{D_Md\rho}{\rho}-(\frac{\triangle_M\rho}{\rho}
  +(p-1)\frac{|d\rho|_M^2}{\rho^2})(g_M+g_N)\\
  &=&\mbox{Ric}(g_M)+(p-2)\frac{D_Md\rho}{\rho}-(\frac{\triangle_M\rho}{\rho}+(p-1)\frac{|d\rho|_M^2}{\rho^2})g_M\\
  &&+\mbox{Ric}(g_N)-(\frac{\triangle_M\rho}{\rho}
  +(p-1)\frac{|d\rho|_M^2}{\rho^2})g_N,
  \end{eqnarray*}
where $D_M$ and $\triangle_M$ are respectively the Levi-Civita connection and the positive Laplace operator of $g_M$. We want to determine when $g$ is an Einstein metric, i.e.,
\begin{align*}
\mbox{Ric}(g)=\lambda g=\lambda\rho^{-2}g_M+\lambda\rho^{-2}g_N
\end{align*}
for some constant $\lambda\in\mathbb{R}$. There are two cases depending on $n>0$ or $n=0$.

If $n>0$, then the Einstein condition is equivalent to
\begin{align}
\label{pre1}\mbox{Ric}(g_M)+(p-2)\frac{D_Md\rho}{\rho}-(\frac{\triangle_M\rho}{\rho}+(p-1)\frac{|d\rho|_M^2}{\rho^2})g_M
=\frac{\lambda}{\rho^2}g_M,&&\\
\label{pre2}\mbox{Ric}(g_N)-(\frac{\triangle_M\rho}{\rho}+(p-1)\frac{|d\rho|_M^2}{\rho^2})g_N=\frac{\lambda}{\rho^2}g_N.&&
\end{align}
It follows from \eqref{pre2} that
\begin{align*}
\mbox{Ric}(g_N)=(\frac{\triangle_M\rho}{\rho}+(p-1)\frac{|d\rho|_M^2}{\rho^2}+\frac{\lambda}{\rho^2})g_N.
\end{align*}
The scalar quantity in the parenthesis on the right-hand side is a constant when restricted to $N$. Hence $(N,g_N)$ is necessarily an Einstein manifold with Einstein constant
\begin{align*}
\epsilon=\frac{\triangle_M\rho}{\rho}+(p-1)\frac{|d\rho|_M^2}{\rho^2}+\frac{\lambda}{\rho^2}.
\end{align*}
From this formula, we get
\begin{align}
\label{ec}\lambda=\epsilon\rho^2-\rho\triangle_M\rho-(p-1)|d\rho|_M^2.
\end{align}
Substituting \eqref{ec} into \eqref{pre1} leads to
\begin{align}
\label{ce}\mbox{Ric}(g_M)+(p-2)\frac{D_Md\rho}{\rho}=\epsilon g_M.
\end{align}
Conversely, assume $(N,g_N)$ is an Einstein manifold with Einstein constant $\epsilon$, and $\rho$ is a positive smooth function on $M$ which satisfies \eqref{ce}. Then the Ricci tensor of $g$ is
\begin{align*}
\mbox{Ric}(g)=(\epsilon-\frac{\triangle_M\rho}{\rho}-(p-1)\frac{|d\rho|_M^2}{\rho^2})(g_M+g_N)=(\epsilon\rho^2-\rho\triangle_M\rho
-(p-1)|d\rho|_M^2)g.
\end{align*}
Hence $g$ is an Einstein metric as $p>2$, with Einstein constant $\lambda$ given by \eqref{ec}.

If $n=0$, i.e., $N$ is a point, then the latter part of the above arguments is still valid. More precisely, if $\rho$ is a positive smooth function on $M$ such that \eqref{ce} holds true for some constant $\epsilon$, then $g$ is an Einstein metric with Einstein constant $\lambda$ given by \eqref{ec}. However, this is in general only a sufficient condition for $g=\rho^{-2}g_M$ being an Einstein metric.

To sum up, we have shown (compare Proposition \ref{eif}.1 to \cite[Theorem 1(I)]{Cle08})
\begin{proposition}\label{eif}
Let $(M^m,g_M)$ and $(N^n,g_N)$, $p=m+n>2$, be two smooth connected Riemannian manifolds, and $\rho$ be a positive smooth function on $M$.
\begin{enumerate}
  \item If $n>0$, then $g=\rho^{-2}(g_M+g_N)$ is an Einstein metric iff $g_N$ is an Einstein metric with Einstein constant $\epsilon$, and \eqref{ce} holds.
  \item If $n=0$, then $g=\rho^{-2}g_M$ is an Einstein metric if \eqref{ce} holds for some constant $\epsilon$.
\end{enumerate}
In either case, the Einstein constant $\lambda$ of $g$ is given by \eqref{ec}.
\end{proposition}

When $n>0$, Eq. \eqref{ce} is equivalent to the quasi-Einstein equation \cite{CasShuWei08}
\begin{align}
\label{qee}\mbox{Ric}(\check{g}_M)-n\frac{\check{D}_Md\check{\rho}}{\check{\rho}}=\lambda\check{g}_M,
\end{align}
with $\check{g}_M=\rho^{-2}g_M$ and $\check{\rho}=\rho^{-1}$. Recall that the warped product manifold $(M\times N,\check{g}_M+\check{\rho}^2g_N)$ is an Einstein manifold with Einstein constant $\lambda$ iff \eqref{qee} holds true, and $(N,g_N)$ is an Einstein manifold with some suitable choice of Einstein constant $\epsilon$ \cite[Corollary 9.107]{Bes87}. Since $\check{g}_M+\check{\rho}^2g_N=\rho^{-2}(g_M+g_N)$, we thus conclude that \eqref{qee} holds for $(\check{g}_M,\check{\rho})$ iff \eqref{ce} holds for $(g_M,\rho)$.
\begin{remark}
Some examples of complete positive Einstein manifolds of the warped product type can be found in \cite[\S4]{LvPagPop04}.
\end{remark}

When $n=0$, \eqref{ce} holds for $(g_M,\rho)$ iff $g=\rho^{-2}g_M$ is an Einstein metric with Einstein constant $\lambda$ and the following equation holds for $(g,\breve{\rho})$ with $\breve{\rho}=\log\rho$
\begin{align}\label{pois}
\triangle_g\breve{\rho}=\epsilon\exp(2\breve{\rho})-\lambda.
\end{align}
In case $M$ is closed, we have respectively at the maximum or minimum points for $\breve{\rho}$,
\begin{align*}
0\le\triangle_g\max\breve{\rho}=\epsilon\exp(2\max\breve{\rho})-\lambda,&&0\ge\triangle_g\min\breve{\rho}
=\epsilon\exp(2\min\breve{\rho})-\lambda.
\end{align*}
Hence $\epsilon\exp(2\min\breve{\rho})\le\lambda\le\epsilon\exp(2\max\breve{\rho})$, i.e., $\epsilon\ge0$. If $\epsilon=0$, then $\lambda=0$ and $\triangle_g\breve{\rho}=0$. It follows that $\breve{\rho}$ must be a constant because $M$ is closed. If $\epsilon>0$, then $\lambda>0$. We do not know whether \eqref{pois} admits nontrivial solutions in this situation.

\subsection{Regularity}\label{regu}

The rest of this section is devoted to discussing differentiability of solutions of the quasi-Einstein-like equation \eqref{ce}. Roughly speaking, we will prove that any $C^{1,1}$ solution of \eqref{ce} is real analytic, which is an analogue of \cite[Lemma 2.2]{DanWan08} for the Ricci soliton equation. It enables us later to conclude that the conformally Einstein metrics constructed in \S\ref{constr} are indeed real analytic after we check that they are $C^2$ in some smooth coordinate charts.

We rewrite \eqref{ce} in the form
\begin{align}
\label{nnce}Dd\rho+\frac{\rho}{p-2}(\mbox{Ric}(g)-\epsilon g)=0,
\end{align}
where we have temporarily dropped the subscript $M$ for simplicity.
\begin{lemma}
\begin{align}
\label{sscbi}\mbox{Ric}(D\log\rho)=\frac{1}{2p-2}DR+\frac{R-(m-1)\epsilon}{p-1}D\log\rho,
\end{align}
where $R$ is the scalar curvature of $g$, and by abuse of notation, $\mbox{Ric}$ also stands for the Ricci endomorphism associated to the Ricci tensor of $g$.
\end{lemma}
\begin{proof}
Given $x\in M$, let $E_1,\cdots,E_m$ be a geodesic frame at $x$. The covariant differential of \eqref{nnce}, when evaluated at $x$, gives
\begin{eqnarray*}
D^2d\rho(E_i,E_j,E_k)&=&-\frac{1}{p-2}D(\rho(\mbox{Ric}(g)-\epsilon g))(E_i,E_j,E_k)\\
&=&-\frac{1}{p-2}(E_k(\rho)(\mbox{Ric}(E_i,E_j)-\epsilon g(E_i,E_j))+\rho D\mbox{Ric}(E_i,E_j,E_k)).\\
\end{eqnarray*}
The Ricci identity says
\begin{align*}
D^2d\rho(E_i,E_j,E_k)-D^2d\rho(E_i,E_k,E_j)=\mbox{Riem}(E_j,E_k,D\rho,E_i),
\end{align*}
where $\mbox{Riem}$ is the Riemann curvature tensor of $g$. Hence
\begin{eqnarray*}
\mbox{Ric}(E_j,D\rho)&=&\sum_{i=1}^m\mbox{Riem}(E_j,E_i,D\rho,E_i)\\
&=&\sum_{i=1}^m(D^2d\rho(E_i,E_j,E_i)-D^2d\rho(E_i,E_i,E_j))\\
&=&-\frac{1}{p-2}\sum_{i=1}^m(E_i(\rho)(\mbox{Ric}(E_i,E_j)-\epsilon g(E_i,E_j))+\rho D\mbox{Ric}(E_i,E_j,E_i)\\
&&-E_j(\rho)(\mbox{Ric}(E_i,E_i)-\epsilon g(E_i,E_i))-\rho D\mbox{Ric}(E_i,E_i,E_j))\\
&=&-\frac{1}{p-2}(\mbox{Ric}(D\rho,E_j)-\epsilon E_j(\rho)-\rho\delta\mbox{Ric}(E_j)-E_j(\rho)(R-m\epsilon)-\rho E_j(R))\\
&=&-\frac{1}{p-2}(\mbox{Ric}(E_j,D\rho)-E_j(\rho)(R-(m-1)\epsilon)-\frac{1}{2}\rho E_j(R)),
\end{eqnarray*}
where $\delta$ is the divergence operator with respect to $g$, and the contracted differential Bianchi identity $\delta\mbox{Ric}=-\frac{1}{2}dR$ has been used in the last equality. Thus we have
\begin{align*}
\mbox{Ric}(E_j,D\rho)=\frac{R-(m-1)\epsilon}{p-1}E_j(\rho)+\frac{\rho}{2p-2}E_j(R).
\end{align*}
The formula \eqref{sscbi} then follows from dividing by $\rho$ both sides of this equation.
\end{proof}

For convenience, let $v=(p-2)\log\rho$, i.e., $\rho=\exp(\frac{v}{p-2})$. It is easy to see that
\begin{eqnarray*}
\frac{d\rho}{\rho}=\frac{dv}{p-2},&&\frac{Dd\rho}{\rho}=\frac{Ddv}{p-2}+\frac{dv\otimes dv}{(p-2)^2}.
\end{eqnarray*}
Then \eqref{nnce} and \eqref{sscbi} become
\begin{eqnarray}
&&\label{nnnce}\mbox{Ric}(g)=\epsilon g-Ddv-\frac{dv\otimes dv}{p-2},\\
&&\label{nsscb}\mbox{Ric}(Dv)=\frac{p-2}{2p-2}DR+\frac{R-(m-1)\epsilon}{p-1}Dv.
\end{eqnarray}

\begin{lemma}
\begin{align}
\label{csbi}\triangle dv=\frac{2p-2}{p-2}dv\circ\mbox{Ric}+\frac{1}{p-2}d(|dv|^2)-\frac{2(R-(m-1)\epsilon)}{p-2}dv.
\end{align}
\end{lemma}
\begin{proof}
The trace of \eqref{nnnce} is
\begin{align*}
R=\triangle v-\frac{1}{p-2}|dv|^2+m\epsilon.
\end{align*}
We differentiate this equality to get
\begin{align*}
dR=d\triangle v-\frac{1}{p-2}d(|dv|^2)=\triangle dv-\frac{1}{p-2}d(|dv|^2).
\end{align*}
For any vector field $X$ on $M$, it follows from the formula \eqref{nsscb} that
\begin{eqnarray*}
dv\circ\mbox{Ric}(X)&=&\mbox{Ric}(X,Dv)\\
&=&\frac{p-2}{2p-2}dR(X)+\frac{R-(m-1)\epsilon}{p-1}dv(X)\\
&=&\frac{p-2}{2p-2}\triangle dv(X)-\frac{1}{2p-2}d(|dv|^2)(X)+\frac{R-(m-1)\epsilon}{p-1}dv(X).
\end{eqnarray*}
Sorting this out leads to the formula \eqref{csbi}.
\end{proof}
We shall combine \eqref{nnnce} and \eqref{csbi} to form a system for the unknowns $(g,v)$, although the latter equation is a consequence of the former. This system is a special case of a little more general system for the unknowns $(g,\omega)\in S^2T^\ast M\oplus T^\ast M$
\begin{align}
\label{fsf}\mbox{Ric}(g)-\epsilon g+\delta^\ast\omega+\frac{\omega\otimes\omega}{p-2}=0,&&\\
\label{fss}\triangle\omega-\frac{2p-2}{p-2}\omega\circ\mbox{Ric}-\frac{1}{p-2}d(|\omega|^2)
+\frac{2(R-(m-1)\epsilon)}{p-2}\omega=0,&&
\end{align}
where $\delta^\ast$ is the symmetrized covariant derivative with respect to $g$. One important feature of system \eqref{fsf} and \eqref{fss} is that it becomes a quasi-linear elliptic system in harmonic coordinates. This leads to a satifactory regularity result.
\begin{proposition}\label{reg}
Any $C^{1,1}$ solution $(g,\omega)$ of system \eqref{fsf} and \eqref{fss} is real analytic.
\end{proposition}
The local nature of real analyticity makes it safe to work in an arbitrarily small neighborhood of a given point. We proceed to establish a useful lemma \cite[Lemma 1.2]{DeTKaz81}.
\begin{lemma}\label{hc}
Let $g$ be a metric of class $C^{1,1}$ in a coordinate neighborhood of $x$. For any $0<\alpha<1$ and $1<q<\infty$, there exist harmonic coordinates in a sub-neighborhood of $x$, which are $C^{2,\alpha}\cap W^{3,q}$ functions of the original coordinates.
\end{lemma}
\begin{proof}
Let
\begin{align*}
y=(y^1,\cdots,y^m):U\subset M^m\rightarrow V\subset\mathbb{R}^m
\end{align*}
be a local coordinate system in the neighborhood $x\in U$, and $g_{ij}=g(\partial_{y^i},\partial_{y^j})$ with inverse matrix $(g^{ij})$. To find harmonic coordinates, we need to solve the Laplace equation
\begin{align}
\label{he}\triangle u=-\sum_{i,j=1}^m(g^{ij}\partial^2_{y^iy^j}u+\frac{1}{\sqrt{G}}\partial_{y^i}(\sqrt{G}g^{ij})\partial_{y^j}u)=0,
\end{align}
where $G$ is the determinant of the positive-definite matrix $(g_{ij})$. Without loss of generality, assume $V$ is a smooth bounded domain, and $g_{ij}\in C^{1,1}(\bar{V})$. In particular, $g_{ij}\in C^{1,\alpha}(\bar{V})$, and hence the coefficients of \eqref{he} belong to $C^{0,\alpha}(\bar{V})$. By \cite[Theorem 6.13]{GilTru01}, \eqref{he} has solutions $z^i\in C^{2,\alpha}(V)$, $1\le i\le m$, which satisfy $\partial_{y^j}z^i(y(x))=\delta_j^i$. Since $g_{ij}\in C^{1,\alpha}(\bar{V})\cap C^{1,1}(\bar{V})$, we now apply \cite[Theorem 9.19]{GilTru01} to conclude that $z^i\in C^{2,\alpha}(\bar{V})\cap W^{3,q}(V)$. Notice that the Jacobian matrix $(\partial_{y^j}z^i)$ is nonsingular at $y(x)$. The inverse function theorem thus asserts a sub-neighborhood $y(x)\in\hat{V}\subset V$, in which $z=(z^1,\cdots,z^m)$ is invertible. Finally, let $\hat{U}=y^{-1}(\hat{V})$ and $\hat{W}=z(\hat{V})$, then
\begin{align*}
z\circ y=(z^1(y^1,\cdots,y^m),\cdots,z^m(y^1,\cdots,y^m)):\hat{U}\subset M^m\rightarrow\hat{W}\subset\mathbb{R}^m
\end{align*}
gives the desired local harmonic coordinate system in the sub-neighborhood $x\in\hat{U}$.
\end{proof}

\begin{corollary}\label{co}
Let $g$ be a metric of class $C^{1,1}$ in a coordinate neighborhood of $x$. For any $0<\alpha<1$ and $1<q<\infty$, there exist harmonic coordinates in a sub-neighborhood of $x$, in which $g$ is of class $C^{1,\alpha}\cap W^{2,q}$.
\end{corollary}
\begin{proof}
Let $y=(y^1,\cdots,y^m):U\rightarrow V$ be a local coordinate system in the neighborhood $x\in U$ as before. Lemma \ref{hc} asserts a local harmonic coordinate system $z\circ y=(z^1,\cdots,z^m):\hat{U}\rightarrow\hat{W}$ in the sub-neighborhood $x\in\hat{U}\subset U$ such that $z^i=z^i(y^1,\cdots,y^m)$, $1\le i\le m$, are of class $C^{2,\alpha}(\bar{\hat{V}})\cap W^{3,q}(\hat{V})$, where $\hat{V}=y(\hat{U})$ is a smooth bounded domain in $\mathbb{R}^m$. It follows from the inverse function theorem that $y^i=y^i(z^1,\cdots,z^m)$, $1\le i\le m$, are of class $C^{2,\alpha}(\bar{\hat{W}})\cap W^{3,q}(\hat{W})$. Since $g(\partial_{y^j},\partial_{y^i})\in C^{1,1}(\bar{\hat{V}})$, the first derivatives $\partial_{y^k}g(\partial_{y^j},\partial_{y^i})$ are Lipschitz continuous, and therefore are differentiable almost everywhere. Furthermore, the second derivatives $\partial^2_{y^ly^k}g(\partial_{y^j},\partial_{y^i})$ are essentially bounded. Notice that
\begin{align*}
g(\partial_{z^i},\partial_{z^j})=\sum_{k,l=1}^mg(\partial_{y^k},\partial_{y^l})\partial_{z^i}y^k\partial_{z^j}y^l.
\end{align*}
It is thus easy to see $g(\partial_{z^i},\partial_{z^j})\in C^{1,\alpha}(\bar{\hat{W}})\cap W^{2,q}(\hat{W})$.
\end{proof}

\begin{remark}\label{re}
If a one-form $\omega$ is of class $C^{1,1}$ in the original coordinate neighborhood, then a similar argument shows that $\omega$ is of class $C^{1,\alpha}\cap W^{2,q}$ in the same harmonic coordinates as in Corollary \ref{co} \cite[Corollary 1.4]{DeTKaz81}.
\end{remark}

\begin{proof}[Proof of Proposition \ref{reg}]
Given $x\in M$, for any $0<\alpha<1$ and $1<q<\infty$, Corollary \ref{co} asserts a local harmonic coordinate system $z=(z^1,\cdots,z^m):\hat{U}\rightarrow\hat{W}$ in a neighborhood $x\in\hat{U}$ such that $g_{ij}=g(\partial_{z^i},\partial_{z^j})\in C^{1,\alpha}(\bar{\hat{W}})\cap W^{2,q}(\hat{W})$. By Remark \ref{re}, $\omega_i=\omega(\partial_{z^i})\in C^{1,\alpha}(\bar{\hat{W}})\cap W^{2,q}(\hat{W})$ as well. We proceed to improve the smoothness of $(g,\omega)$ by turns.

Notice that in harmonic coordinates $(z^1,\cdots,z^m)$, the principal part of \eqref{fsf} is given by $-\frac{1}{2}\sum_{k,l=1}^mg^{kl}\partial^2_{z^kz^l}g_{ij}$ \cite[(4.3)]{DeTKaz81}, which indicates that \eqref{fsf} is an elliptic equation for the unknowns $g_{ij}$. Since $g_{ij}\in W^{2,q}(\hat{W})$, and the coefficients of \eqref{fsf} belong to $C^{0,\alpha}(\bar{\hat{W}})$, we apply \cite[Theorem 9.19]{GilTru01} to conclude that $g_{ij}\in C^{2,\alpha}(\bar{\hat{W}})$. In particular, $\mbox{Ric}(g)\in C^{0,\alpha}(\bar{\hat{W}})$. We turn now to \eqref{fss}, whose principal part is given by $-\sum_{j,k=1}^mg^{jk}\partial^2_{z^jz^k}\omega_i$. Hence \eqref{fss} is an elliptic equation for the unknowns $\omega_i$. Since $\omega_i\in W^{2,q}(\hat{W})$, and the coefficients of \eqref{fss} belong to $C^{0,\alpha}(\bar{\hat{W}})$, we apply again \cite[Theorem 9.19]{GilTru01} to conclude that $\omega_i\in C^{2,\alpha}(\bar{\hat{W}})$.

We have shown that $(g,\omega)$ is actually a $C^{2,\alpha}(\bar{\hat{W}})$ solution of system \eqref{fsf} and \eqref{fss}. As a coupled system for the unknowns $(g,\omega)$, the principal symbol of its linearization at a solution is given by $(h,\psi)\mapsto(\frac{1}{2}|\xi|_g^2h,|\xi|_g^2\psi+\Theta_{(g,\omega)}(\xi,h))$, where $(h,\psi)\in S^2T^\ast M\oplus T^\ast M$, $\xi$ is a cotangent vector, and $\Theta_{(g,\omega)}(\xi,h)$ is linear in $h$. It is clear that the principal symbol is an isomorphism for any nonzero $\xi$. Thus the system is a quasi-linear elliptic system, and we apply \cite[Theorem 6.7.6]{Mor66} to conclude that $(g,\omega)$ is indeed real analytic.
\end{proof}

\section{Constructions}\label{constr}

In this section, we will construct a special class of triples $(\bar{M}^m,g_{\bar{M}},\rho)$, $m\ge4$, where $\bar{M}$ is an $m$-dimensional smooth compact manifold with boundary, $g_{\bar{M}}$ is a Riemannian metric on $\bar{M}$, and $\rho$ is a defining function for boundary, such that \eqref{ce} holds true in the interior of $\bar{M}$ for some constants $(m\le)p\in\mathbb{Z}$ and $\epsilon\in\mathbb{R}$. As before, we introduce a function
\begin{align*}
v=(p-2)\log\rho,
\end{align*}
and then deal with \eqref{nnnce} instead. We first proceed to fix some notations.

\subsection{Basic setup}\label{sta}

Let $\{(V_i^{2n_i},J_i,h_i)\}_{1\le i\le r}$, $r\ge2$, be a finite set of connected Fano K\"ahler-Einstein manifolds with complex dimension $n_i\ge0$. It is well-known that $V_i$ is simply-connected. Therefore $H^2(V_i;\mathbb{Z})$ is torsion-free, and $H^2(V_1\times\cdots\times V_r;\mathbb{Z})\cong H^2(V_1;\mathbb{Z})\oplus\cdots\oplus H^2(V_r;\mathbb{Z})$. Assume the first Chern class $c_1(V_i)=p_ia_i$, where $p_i\in\mathbb{Z}_+$ is a positive integer, and $a_i\in H^2(V_i;\mathbb{Z})$ is an indivisible class. We shall henceforth normalize the K\"ahler-Einstein metric $h_i$ such that $\mbox{Ric}(h_i)=p_ih_i$. Let $\eta_i$ and $\varsigma_i$ be respectively the associated K\"{a}hler form and Ricci form. Then $p_i\eta_i=\varsigma_i=2\pi[c_1(V_i)]=2\pi p_i[a_i]$, i.e., $\eta_i=2\pi[a_i]$, where $[\cdot]$ denotes the corresponding de Rham cohomology class.

Given an ordered $r$-tuple $q=(q_1,\cdots,q_r)$, $q_i\in\mathbb{Z}\backslash\{0\}$, let $\hat{\pi}:P_q\rightarrow V_1\times\cdots\times V_r$ be the principal circle bundle with Euler class $\oplus_{i=1}^rq_ia_i$. Then there exists a principal connection $\theta$ on $P_q$, with curvature form $\Omega=\sum_{i=1}^rq_i\hat{\pi}^\ast\eta_i$. We shall think of the circle as the unitary group $U(1)=\{\exp(\sqrt{-1}\upsilon)|\upsilon\in\mathbb{R}/2\pi\mathbb{Z}\}$ with Lie algebra $\mathfrak{u}(1)=\sqrt{-1}\mathbb{R}$. By convention, $U(1)$ acts freely on $P_q$ from the right. This gives rise to a vertical vector field
\begin{align*}
\Lambda(x)=\frac{d}{d\upsilon}|_{\upsilon=0}(x\cdot\exp(\sqrt{-1}\upsilon)).
\end{align*}
The principal connection $\theta$ is such that $\theta(\Lambda)\equiv1$. Now the relation of the $(2,1)$ O'Neill tensor field $A$ (cf. \cite[9.20]{Bes87}) and the curvature form $\Omega$ becomes
\begin{align}\label{at}
A_YZ=-\frac{1}{2}\Omega(Y,Z)\Lambda,
\end{align}
where $Y$ and $Z$ are arbitrary horizontal vector fields (cf. \cite[9.65]{Bes87}).

We introduce on $P_q$ a one-parameter family of metrics
\begin{eqnarray*}
g_t=f(t)^2\theta\otimes\theta+\sum_{i=1}^rg_i(t)^2\hat{\pi}^\ast h_i,&&t\in I,
\end{eqnarray*}
where $I\subset\mathbb{R}$ is an open interval, $f$ and $g_i$ are positive smooth functions on $I$. For each $t\in I$, the bundle projection map $\hat{\pi}:(P_q,g_t)\rightarrow(V_1\times\cdots\times V_r,\sum_{i=1}^rg_i(t)^2h_i)$ is a Riemannian submersion with totally geodesic fibres, i.e., the $(2,1)$ O'Neill tensor field $T$ (cf. \cite[9.17]{Bes87}) vanishes identically. More importantly, the curvature form $\Omega$ is parallel with respect to any product metric on the base. Therefore the principal connection $\theta$ satisfies the Yang-Mills condition (cf. \cite[9.35 and 9.65]{Bes87}).

Let $\hat{M}^m$ be the product $I\times P_q$ of dimension $m=2+\sum_{i=1}^r2n_i$, which is an open and dense subset of our target manifold $\bar{M}$. Assume the metric $\hat{g}$ on $\hat{M}$ takes the form $dt^2+g_t$, and $v$ depends only on $t\in I$. Notice that the above assumptions are quite natural in view of the geodesic conformal compactifications \eqref{gc}. We need to do some computations in order to write down \eqref{nnnce} for the data $(\hat{M},\hat{g},v)$.

Let $\{\check{E}_{i1},\cdots,\check{E}_{i,2n_i}\}$ be an orthonormal adapted basis on $(V_i,J_i,h_i)$, i.e., $\check{E}_{i,n_i+j}=J_i\check{E}_{ij}$, $1\le j\le n_i$. Let $E_{ij}$ be the unique horizontal vector field on $P_q$ such that $\hat{\pi}_\ast(E_{ij})=\check{E}_{ij}$. Then $\{f(t)^{-1}\Lambda,g_i(t)^{-1}E_{ij}\}_{1\le i\le r,1\le j\le2n_i}$ forms an orthonormal basis on $(P_q,g_t)$. It is straightforward to show that \cite[9.36]{Bes87}
\begin{lemma}\label{nv}
The non-vanishing components of the Ricci tensor of $(P_q,g_t)$ are
\begin{eqnarray*}
\mbox{Ric}_t(\Lambda,\Lambda)=\sum_{i=1}^r\frac{n_iq_i^2}{2}\frac{f^4}{g_i^4},&&\mbox{Ric}_t(E_{ij},E_{ij})=p_i
-\frac{q_i^2}{2}\frac{f^2}{g_i^2}.
\end{eqnarray*}
\end{lemma}
Notice that $\{\partial_t,f(t)^{-1}\Lambda,g_i(t)^{-1}E_{ij}\}_{1\le i\le r,1\le j\le2n_i}$ forms an orthonormal basis on $(\hat{M},\hat{g})$. It is also straightforward to show that \cite[\S2]{EscWan00}
\begin{lemma}
The non-vanishing components of the Ricci tensor of $(\hat{M},\hat{g})$ are
\begin{eqnarray*}
&&\hat{\mbox{Ric}}(\partial_t,\partial_t)=-\frac{\ddot{f}}{f}-\sum_{i=1}^r2n_i\frac{\ddot{g_i}}{g_i},\\
&&\hat{\mbox{Ric}}(\Lambda,\Lambda)=-f\ddot{f}-\sum_{i=1}^r2n_if\dot{f}\frac{\dot{g_i}}{g_i}
+\sum_{i=1}^r\frac{n_iq_i^2}{2}\frac{f^4}{g_i^4},\\
&&\hat{\mbox{Ric}}(E_{ij},E_{ij})=-g_i\ddot{g_i}+\dot{g_i}^2-\frac{\dot{f}}{f}g_i\dot{g_i}
-\sum_{k=1}^r2n_kg_i\dot{g_i}\frac{\dot{g_k}}{g_k}+p_i-\frac{q_i^2}{2}\frac{f^2}{g_i^2},
\end{eqnarray*}
where $\dot{}$ denotes $\frac{d}{dt}$.
\end{lemma}
Furthermore, we have
\begin{lemma}\label{hess}
The non-vanishing components of the Hessian of $v$ with respect to $\hat{g}$ are
\begin{eqnarray*}
\hat{D}dv(\partial_t,\partial_t)=\ddot{v},&\hat{D}dv(\Lambda,\Lambda)=\dot{v}f\dot{f},&\hat{D}dv(E_{ij},E_{ij})=\dot{v}g_i\dot{g_i}.
\end{eqnarray*}
\end{lemma}
It follows that \eqref{nnnce} for the triple $(\hat{M},\hat{g},v)$ is equivalent to the following system of ODEs
\begin{eqnarray*}
\ddot{v}+\frac{\dot{v}^2}{p-2}-\frac{\ddot{f}}{f}-\sum_{i=1}^r2n_i\frac{\ddot{g_i}}{g_i}=\epsilon,&\\
\dot{v}\frac{\dot{f}}{f}-\frac{\ddot{f}}{f}-\sum_{i=1}^r2n_i\frac{\dot{f}\dot{g_i}}{fg_i}
+\sum_{i=1}^r\frac{n_iq_i^2}{2}\frac{f^2}{g_i^4}=\epsilon,&\\
\dot{v}\frac{\dot{g_i}}{g_i}-\frac{\ddot{g_i}}{g_i}+(\frac{\dot{g_i}}{g_i})^2-\frac{\dot{f}\dot{g_i}}{fg_i}
-\sum_{j=1}^r2n_j\frac{\dot{g_i}\dot{g_j}}{g_ig_j}+\frac{p_i}{g_i^2}-\frac{q_i^2}{2}\frac{f^2}{g_i^4}=\epsilon,
&1\le i\le r.
\end{eqnarray*}

\subsection{Exact solutions}\label{es}

Corresponding to the $U(1)$-action on $\hat{M}$, we introduce a moment map coordinate $s$ through $ds=f(t)dt$, and define $\alpha(s)=f(t)^2$, $\beta_i(s)=g_i(t)^2$ and $\varphi(s)=v(t)$. In addition, we let
\begin{align}
\label{w}w=\prod_{i=1}^rg_i^{2n_i}=\prod_{i=1}^r\beta_i^{n_i}.
\end{align}
We can then rewrite the above system of ODEs as follows
\begin{eqnarray}
\label{fi}\alpha(\varphi''+\frac{\varphi'^2}{p-2}
-\sum_{i=1}^rn_i(\frac{\beta_i''}{\beta_i}-\frac{1}{2}(\frac{\beta_i'}{\beta_i})^2))
+\frac{\alpha'}{2}(\varphi'-(\log w)')-\frac{\alpha''}{2}=\epsilon,&\\
\label{se}\frac{\alpha'}{2}(\varphi'-(\log w)')-\frac{\alpha''}{2}+\frac{\alpha}{2}\sum_{i=1}^r\frac{n_iq_i^2}{\beta_i^2}=\epsilon,&\\
\label{th}\frac{\alpha}{2}(\frac{\beta_i'}{\beta_i}(\varphi'-(\log w)')
-\frac{\beta_i''}{\beta_i}+(\frac{\beta_i'}{\beta_i})^2-\frac{q_i^2}{\beta_i^2})-\frac{\alpha'}{2}\frac{\beta_i'}{\beta_i}
+\frac{p_i}{\beta_i}=\epsilon,&1\le i\le r,
\end{eqnarray}
where $'$ denotes $\frac{d}{ds}$. Equating \eqref{fi} with \eqref{se} leads to
\begin{align}
\label{fise}
\varphi''+\frac{\varphi'^2}{p-2}=\sum_{i=1}^rn_i(\frac{\beta_i''}{\beta_i}-\frac{1}{2}(\frac{\beta_i'}{\beta_i})^2
+\frac{q_i^2}{2\beta_i^2}).
\end{align}
In what follows, we will look for exact solutions of system \eqref{fi}-\eqref{th} such that
\begin{eqnarray}
\label{set}
\frac{\beta_i''}{\beta_i}-\frac{1}{2}(\frac{\beta_i'}{\beta_i})^2
+\frac{q_i^2}{2\beta_i^2}=0,&&1\le i\le r.
\end{eqnarray}
\begin{remark}
There is a geometric interpretation of \eqref{set} (cf. \cite[Corollary 7.5]{WanWan98}). One can define a complex structure $\hat{J}$ on $\hat{M}$ by lifting the product
complex structure of the base to the horizontal spaces of $\theta$ and setting
$\hat{J}(\partial_t)=-f^{-1}\Lambda$ on fibers. Then $\hat{g}$ is Hermitian
with respect to $\hat{J}$, and \eqref{set} is equivalent to the fact that the Riemann curvature tensor of $\hat{g}$ is fully invariant under the action of $\hat{J}$.
\end{remark}
Under the assumption \eqref{set}, \eqref{fise} becomes
\begin{align*}
\varphi''+\frac{\varphi'^2}{p-2}=0.
\end{align*}
This Bernoulli equation has a nontrivial solution
\begin{eqnarray}
\label{var}\varphi=(p-2)(\log(s+\kappa_0)+\log\kappa_1),&&s>-\kappa_0,
\end{eqnarray}
where $\kappa_0$ and $\kappa_1(>0)$ are constants of integration. Therefore,
\begin{eqnarray*}
\rho=\exp(\frac{v}{p-2})=\exp(\frac{\varphi}{p-2})=\kappa_1(s+\kappa_0)>0,&&\mbox{for }s>-\kappa_0.
\end{eqnarray*}

Eq. \eqref{set} itself has two types of solutions: the linear solutions $\beta_i=\pm q_i(s+\varpi_i)$ and the quadratic solutions
\begin{align}\label{beta}
\beta_i=A_i(s+s_i)^2-\frac{q_i^2}{4A_i},
\end{align}
where $\varpi_i$, $A_i(\ne0)$ and $s_i$ are constants of integration.
\begin{claim}
System \eqref{fi}-\eqref{th} admits no solutions with linear $\beta_i$'s.
\end{claim}
\begin{proof}
We substitute $\beta_i=\iota_i(s+\varpi_i)$, $\iota_i=\pm q_i$, into \eqref{se} and \eqref{th} to get
\begin{eqnarray}
&&\label{nfi}\alpha''=\alpha'(\varphi'-(\log w)')+\alpha\sum_{i=1}^r\frac{n_iq_i^2}{\beta_i^2}-2\epsilon,\\
&&\label{nse}\alpha'=\alpha(\varphi'-(\log w)')+\frac{2p_i}{\iota_i}-2\epsilon(s+\varpi_i),\;\;1\le i\le r.
\end{eqnarray}
Differentiating \eqref{nse} with respect to $s$ gives
\begin{align}
\label{dnse}\alpha''=\alpha'(\varphi'-(\log w)')+\alpha(\varphi''-(\log w)'')-2\epsilon.
\end{align}
Equating \eqref{nfi} with \eqref{dnse}, and noticing $(\log w)''=-\sum_{i=1}^r\frac{n_iq_i^2}{\beta_i^2}$, we end up with $\varphi''=0$. This gives a contradiction.
\end{proof}
Hence from now on, we shall focus on the quadratic case. In order to solve for $\alpha$, we substitute the expression of $\beta_i$ given by \eqref{beta} into \eqref{th} to get
\begin{eqnarray}
\label{qth}\alpha'=\alpha(\varphi'-(\log w)'+\frac{1}{s+s_i})-\epsilon(s+s_i)+\frac{1}{s+s_i}(\frac{\epsilon q_i^2}{4A_i^2}+\frac{p_i}{A_i}),&&1\le i\le r.
\end{eqnarray}
This gives rise to two consistency conditions, i.e.,
\begin{eqnarray*}
s_i\equiv s_0\mbox{ and }\frac{\epsilon q_i^2}{4A_i^2}+\frac{p_i}{A_i}\equiv\nu,&&1\le i\le r
\end{eqnarray*}
for some constants $s_0$ and $\nu$. In this case, \eqref{qth} becomes
\begin{align}
\label{cqth}\alpha'=\alpha(\varphi'-(\log w)'+\frac{1}{s+s_0})-\epsilon(s+s_0)+\frac{\nu}{s+s_0}.
\end{align}
Differentiating \eqref{cqth} with respect to $s$ gives
\begin{align*}
\alpha''=\alpha'(\varphi'-(\log w)'+\frac{1}{s+s_0})+\alpha(\varphi''-(\log w)''-\frac{1}{(s+s_0)^2})-\epsilon-\frac{\nu}{(s+s_0)^2}.
\end{align*}
Substituting this equation into \eqref{se} gives
\begin{align}
\label{sdc}\frac{\alpha'}{s+s_0}+\alpha(\varphi''-(\log w)''-\frac{1}{(s+s_0)^2}-\sum_{i=1}^r\frac{n_iq_i^2}{\beta_i^2})+\epsilon-\frac{\nu}{(s+s_0)^2}=0.
\end{align}
Finally, we substitute \eqref{cqth} into \eqref{sdc} to get
\begin{align*}
\frac{\varphi'}{s+s_0}-\frac{(\log w)'}{s+s_0}+\varphi''-(\log w)''-\sum_{i=1}^r\frac{n_iq_i^2}{\beta_i^2}=0,
\end{align*}
which simplifies to
\begin{align*}
\frac{(p-2)(\kappa_0-s_0)}{(s+s_0)(s+\kappa_0)^2}=0
\end{align*}
once we take into consideration \eqref{w}, \eqref{var}, \eqref{beta} and the consistency condition $s_i\equiv s_0$. It follows that $s_0=\kappa_0$. Now the consistency conditions for system \eqref{fi}-\eqref{th} become
\begin{eqnarray}
\label{csc}
s_i\equiv\kappa_0\mbox{ and }\frac{\epsilon q_i^2}{4A_i^2}+\frac{p_i}{A_i}\equiv\nu,&&1\le i\le r.
\end{eqnarray}

Under these consistency conditions, system \eqref{fi}-\eqref{th} reduces to a linear equation
\begin{align}
\label{sie}\alpha'=\alpha(\frac{p-1}{s+\kappa_0}-(\log w)')-\epsilon(s+\kappa_0)+\frac{\nu}{s+\kappa_0},
\end{align}
whose solution is given by
\begin{align}
\label{alp}\alpha=\frac{(s+\kappa_0)^{p-1}}{w(s)}\int^s_{s_\ast}\frac{w(\tau)(\nu-\epsilon(\tau
+\kappa_0)^2)}{(\tau+\kappa_0)^p}d\tau,
\end{align}
where $s_\ast$ denotes the right end-point of the interval $s\in I$.

To sum up, we have proved
\begin{proposition}\label{loc}
Formulas \eqref{alp}, \eqref{beta}, \eqref{var}, \eqref{w} together with consistency condition \eqref{csc} give us an exact solution of system \eqref{fi}-\eqref{th} on the interval $I=(-\kappa_0,s_\ast)$, with arbitrary constants $\kappa_0$, $\kappa_1(>0)$, $A_i(\ne0)$, $\nu$ and $s_\ast(>-\kappa_0)$. Furthermore, each solution corresponds to a local Riemannian metric
\begin{align*}
\hat{g}=dt^2+f(t)^2\theta\otimes\theta+\sum_{i=1}^rg_i(t)^2\hat{\pi}^\ast h_i=\alpha(s)^{-1}ds^2+\alpha(s)\theta\otimes\theta+\sum_{i=1}^r\beta_i(s)\hat{\pi}^\ast h_i
\end{align*}
on $\hat{M}=(-\kappa_0,s_\ast)\times P_q$, where $P_q$ is the principal circle bundle with principal connection $\theta$.
\end{proposition}

\subsection{Compactifications}\label{com}

In order to extend the open Riemannian manifold $(\hat{M},\hat{g})$ in Proposition \ref{loc} smoothly to a compact manifold with boundary, we need to compactify one of its ends. To achieve this goal, let $V_1=\mathbb{C}P^{n_1}$, $n_1\ge0$, and we shall add to $\hat{M}$ a copy of $V_2\times\cdots\times V_r$ at $s=s_\ast$, called the event horizon, which is a deformation retraction of $\hat{M}$. This procedure corresponds to blowing down the standard Hopf fibration $S^{2n_1+1}\rightarrow\mathbb{C}P^{n_1}$ to a point at $s=s_\ast$. Therefore we need $|q_1|=1$. The corresponding smoothness conditions are (cf. \cite[p. 16]{DanWan08}): near $t=t_\ast=t(s_\ast)$, $f$ is smooth and odd in $t$ with $\dot{f}(t_\ast)=-1$; $g_1$ is smooth and odd in $t$ with $\dot{g_1}(t_\ast)^2=\frac{1}{2}$; $g_i$ is smooth and even in $t$ with $g_i(t_\ast)>0$, $2\le i\le r$. Since $\frac{d}{dt}=\sqrt{\alpha}\frac{d}{ds}$, $f=\sqrt{\alpha}$, $g_i=\sqrt{\beta_i}$ and $v=\varphi$, it follows that if near $s=s_\ast$, $\alpha$ is a $C^2$ function in $s$ with $\alpha(s_\ast)=0$ and $\alpha'(s_\ast)=-2$; $\beta_1$ is a $C^2$ function in $s$ with $\beta_1(s_\ast)=0$ and $\beta_1'(s_\ast)=-1$; $\beta_i$ is a $C^2$ function in $s$ with $\beta_i(s_\ast)>0$, $2\le i\le r$; $\varphi$ is a $C^3$ function in $s$, then the above smoothness conditions hold up to order $2$. By Proposition \ref{reg}, $(\hat{M},\hat{g},v)$ thus extends smoothly to $s=s_\ast$.

We proceed to simplify our formulas. The constant $\kappa_0$ in Proposition \ref{loc} represents the translation invariance in $s$ of the defining function $\rho$. Without loss of generality, we may assume $\kappa_0=0$, i.e., $\rho=\kappa_1s$, and hence $\rho(0)=0$. We may also assume that $s=0$ corresponds to $t=0$, and therefore the interval $I$ takes the form $(0,t_\ast)$.

Now \eqref{alp}, \eqref{beta} and \eqref{var} become
\begin{eqnarray}
&&\label{cal}\alpha(s)=\frac{s^{p-1}}{w(s)}\int_s^{s_\ast}\frac{w(\tau)(\epsilon\tau^2-\nu)}{\tau^p}d\tau,\\
&&\label{cbe}\beta_i(s)=A_i(s^2-\frac{q_i^2}{4A_i^2}),\\
&&\label{cva}\varphi(s)=(p-2)(\log s+\log\kappa_1).
\end{eqnarray}
It follows from \eqref{cbe} and \eqref{cva} that $\beta_i$ and $\varphi$ are analytic functions of $s\in(0,s_\ast)$.

\begin{proposition}\label{ral}
$\alpha$ is a rational function of $s\in(0,s_\ast)$.
\end{proposition}
\begin{proof}
It follows from \eqref{w} and \eqref{cbe} that
\begin{align*}
w(s)=\prod_{i=1}^r(A_i(s^2-\frac{q_i^2}{4A_i^2}))^{n_i}=C_1Q(s),
\end{align*}
where $C_1=\prod_{i=1}^rA_i^{n_i}$, and $Q(s)=\sum_{j=0}^{m/2-1}a_js^{2j}$ with
\begin{eqnarray}\label{aj}
a_j=\sum_{k_1+\cdots+k_r=j,\;0\le k_i\le n_i}\prod_{i=1}^rC_{n_i}^{k_i}(-\frac{q_i^2}{4A_i^2})^{n_i-k_i},&&0\le j\le\frac{m}{2}-1=\sum_{i=1}^rn_i.
\end{eqnarray}
In particular, $Q(0)=a_0=\prod_{i=1}^r(-\frac{q_i^2}{4A_i^2})^{n_i}\ne0$ and $a_{m/2-1}=1$. Furthermore,
\begin{align*}
Q(s)(\epsilon s^2-\nu)=\sum_{j=0}^{m/2}b_js^{2j},
\end{align*}
where
\begin{eqnarray}\label{b}
b_0=-\nu a_0;&b_j=\epsilon a_{j-1}-\nu a_j,\;1\le j\le\frac{m}{2}-1;&b_{m/2}=\epsilon a_{m/2-1}=\epsilon.
\end{eqnarray}
Substituting these into \eqref{cal} leads to
\begin{align*}
\alpha(s)=\frac{s^{p-1}}{Q(s)}\int_s^{s_\ast}\frac{Q(\tau)(\epsilon\tau^2-\nu)}{\tau^p}d\tau
=\frac{s^{p-1}}{Q(s)}\sum_{j=0}^{m/2}b_j\int_s^{s_\ast}\tau^{2j-p}d\tau.
\end{align*}
If $2j_\ast-p+1=0$ for some integer $0\le j_\ast\le\frac{m}{2}$, then $j_\ast=\frac{p-1}{2}\le\frac{m}{2}$, i.e., $p=m+n\le m+1$. Hence $n\le1$. Since $j_\ast$ is an integer, $p$ has to be odd. So does $n$ as $m=2+\sum_{i=1}^r2n_i$ is always even. Hence $n=1$ and $\epsilon=0$ because the (flat) circle is the only $1$-dimensional closed Riemannian manifold. It follows that $j_\ast=\frac{m}{2}$ and $b_{m/2}=0$. Hence no logarithmic term appears after integration, and
\begin{align}\label{alpq}
\alpha(s)=\frac{s^{p-1}}{Q(s)}\sum_{j=0}^{m/2}\frac{b_j}{2j-p+1}\tau^{2j-p+1}|_s^{s_\ast}=\frac{P(s)}{Q(s)},
\end{align}
where $P(s)=C_2s^{p-1}-\sum_{j=0}^{m/2}\frac{b_j}{2j-p+1}s^{2j}$ with $C_2=\sum_{j=0}^{m/2}\frac{b_j}{2j-p+1}s_\ast^{2j-p+1}$. Therefore, $\alpha$ is a rational function of $s$ because both $P$ and $Q$ are polynomial functions of $s$.
\end{proof}
\begin{remark}\label{ap}
Since $p\ge4$ by our constructions, it follows easily from \eqref{alpq} that $\alpha'(0)=0$.
\end{remark}

In order to guarantee that the boundary metric
\begin{align*}
g_0=\alpha(0)\theta\otimes\theta+\sum_{i=1}^r\beta_i(0)\hat{\pi}^\ast h_i
\end{align*}
is non-degenerate, we shall impose the boundary conditions that $\alpha(0)>0$ and $\beta_i(0)>0$, $1\le i\le r$. Since $\beta_i(0)=-\frac{q_i^2}{4A_i}$, we see $\beta_i(0)>0$ iff $A_i<0$. Thus $\beta_i$ is concave, and $\beta_i>0$ in $(0,s_\ast)$ provided $\beta_i(s_\ast)\ge0$. Consequently, $w>0$ on $[0,s_\ast)$. By l'H\^{o}pital's rule, we have
\begin{eqnarray*}
\lim_{s\searrow0}\alpha(s)&=&\frac{1}{w(0)}\lim_{s\searrow0}\frac{1}{s^{1-p}}\int_s^{s_\ast}\frac{w(\tau)(\epsilon\tau^2
-\nu)}{\tau^p}d\tau\\
&=&\frac{1}{w(0)}\lim_{s\searrow0}\frac{1}{(1-p)s^{-p}}\cdot\frac{w(s)(\nu-\epsilon s^2)}{s^p}\\
&=&\frac{\nu}{1-p}\\
&>&0,
\end{eqnarray*}
provided $\nu<0$. It remains to show the positivity of $\alpha$ on $[0,s_\ast)$ (see Proposition \ref{poal} below), which turns out to be a consequence of the smooth compactification at $s=s_\ast$. After this is done, we conclude that
\begin{align*}
g_s=\alpha(s)\theta\otimes\theta+\sum_{i=1}^r\beta_i(s)\hat{\pi}^\ast h_i
\end{align*}
is non-degenerate on $[0,s_\ast)$, and $\hat{g}=\alpha(s)^{-1}ds^2+g_s$ is smooth until $s=0$.

In what follows, we proceed to analyze the set of smooth compactification conditions obtained at the beginning of this subsection.

First of all, it follows from $\beta_1(s_\ast)=0$ that $s_\ast^2=\frac{q_1^2}{4A_1^2}$, i.e., $s_\ast=-\frac{1}{2A_1}>0$ as $|q_1|=1$ and $A_1<0$. In particular, $\beta_1'(s_\ast)=2A_1s_\ast=-1$.
\begin{lemma}\label{poi}
$\epsilon s^2-\nu>0$, $\forall\;s\in[0,s_\ast]$.
\end{lemma}
\begin{proof}
This is obvious for $\epsilon\ge0$ as $\nu<0$. For $\epsilon<0$, we have
\begin{align*}
\epsilon s^2-\nu\ge\epsilon s_\ast^2-\nu=\frac{\epsilon q_1^2}{4A_1^2}-\nu=-\frac{p_1}{A_1}>0,
\end{align*}
where the consistency condition \eqref{csc} has been used in the last equality.
\end{proof}
\begin{proposition}\label{poal}
$\alpha(s)>0$, $\forall\;s\in[0,s_\ast)$.
\end{proposition}
\begin{proof}
The statement follows from \eqref{cal}, the positivity of $w$ in $[0,s_\ast)$ and Lemma \ref{poi}.
\end{proof}
Secondly, it follows from $\beta_i(s_\ast)>0$ that $s_\ast^2<\frac{q_i^2}{4A_i^2}$, i.e.,
\begin{eqnarray}
\label{c3}A_i^2<q_i^2A_1^2,&&2\le i\le r.
\end{eqnarray}
\begin{proposition}\label{alc}
$\alpha(s_\ast)=0$ and $\alpha'(s_\ast)=-2$.
\end{proposition}
\begin{proof}
To show that $\alpha$ vanishes at $s=s_\ast$, we shall distinguish two cases.

If $n_1=0$, then $w(s_\ast)>0$, and it follows directly from \eqref{cal} that $\alpha(s_\ast)=0$.

If $n_1>0$, then $w(s_\ast)=0$ as $\beta_1(s_\ast)=0$. We have to appeal to l'H\^{o}pital's rule to deduce
\begin{align*}
\lim_{s\nearrow s_\ast}\alpha(s)=s_\ast^{p-1}\lim_{s\nearrow s_\ast}\frac{1}{w'(s)}\cdot\frac{w(s)(\nu-\epsilon s^2)}{s^p}=\frac{\nu-\epsilon s_\ast^2}{s_\ast}\lim_{s\nearrow s_\ast}\frac{w(s)}{w'(s)}=0,
\end{align*}
as $\beta_1'(s_\ast)\ne0$. In both cases, we end up with $\alpha(s_\ast)=0$. By definition, we have
\begin{eqnarray*}
\alpha'(s_\ast)&=&\lim_{s\nearrow s_\ast}\frac{\alpha(s)}{s-s_\ast}\\
&=&s_\ast^{p-1}\lim_{s\nearrow s_\ast}\frac{1}{(s-s_\ast)w(s)}\int_s^{s_\ast}\frac{w(\tau)(\epsilon\tau^2-\nu)}{\tau^p}d\tau\\
&=&s_\ast^{p-1}\lim_{s\nearrow s_\ast}\frac{1}{w(s)+(s-s_\ast)w'(s)}\cdot\frac{w(s)(\nu-\epsilon s^2)}{s^p}\\
&=&\frac{\nu-\epsilon s_\ast^2}{s_\ast}\lim_{s\nearrow s_\ast}\frac{1}{1+(s-s_\ast)(\log w(s))'}\\
&=&\frac{\nu-\epsilon s_\ast^2}{s_\ast}\cdot\frac{1}{n_1+1}\\
&=&-2.
\end{eqnarray*}
In the above we have used the facts $\beta_1=A_1(s^2-s_\ast^2)$ and $p_1=n_1+1$ for $V_1=\mathbb{C}P^{n_1}$.
\end{proof}
It remains to determine $A_i$ from the consistency condition \eqref{csc}, and then make sure that they satisfy the inequality \eqref{c3}. We shall divide our discussion into three cases $\epsilon=0$, $\epsilon<0$ and $\epsilon>0$.

When $\epsilon=0$, we have simply
\begin{align*}
A_i=\frac{p_i}{\nu}.
\end{align*}
Substituting this into \eqref{c3} leads to a topological restriction $p_i^2<p_1^2q_i^2$, i.e., $p_1|q_i|>p_i$, $2\le i\le r$.

If $\epsilon\ne0$, we can then rewrite the equation for $A_i$ as follows
\begin{align}
\label{qu}A_i^2-\frac{p_i}{\nu}A_i-\frac{\epsilon q_i^2}{4\nu}=0,
\end{align}
whose discriminant is given by $\Delta_i=\frac{p_i^2}{\nu^2}+\frac{\epsilon q_i^2}{\nu}$. We need $\Delta_i\ge0$ to guarantee that \eqref{qu} has real (negative) roots.

When $\epsilon<0$, $\Delta_i$ is always positive as $\nu\epsilon>0$, and \eqref{qu} has a unique negative root
\begin{align*}
A_i=\frac{1}{2}(\frac{p_i}{\nu}-\sqrt{\Delta_i})=\frac{1}{2\nu}(p_i+\sqrt{p_i^2+\nu\epsilon q_i^2}).
\end{align*}
Substituting this into \eqref{c3} leads to the same topological restriction $p_1|q_i|>p_i$, $2\le i\le r$.

When $\epsilon>0$, $\Delta_i\ge0$ is equivalent to $\nu\ge-\frac{p_i^2}{\epsilon q_i^2}$, $1\le i\le r$. Now \eqref{qu} has two types of negative roots
\begin{align*}
A_i=\frac{1}{2}(\frac{p_i}{\nu}\pm\sqrt{\Delta_i})=\frac{1}{2\nu}(p_i\mp\sqrt{p_i^2+\nu\epsilon q_i^2}).
\end{align*}
We shall substitute these into \eqref{c3} to deduce the (topological) restrictions. There are four cases as follows.

If $A_1=\frac{1}{2\nu}(p_1+\sqrt{p_1^2+\nu\epsilon q_1^2})$ and $A_i=\frac{1}{2\nu}(p_i+\sqrt{p_i^2+\nu\epsilon q_i^2})$, then the resulting topological restriction is $p_1|q_i|>p_i$.

If $A_1=\frac{1}{2\nu}(p_1+\sqrt{p_1^2+\nu\epsilon q_1^2})$ and $A_i=\frac{1}{2\nu}(p_i-\sqrt{p_i^2+\nu\epsilon q_i^2})$, then the resulting restriction is $\nu>-\frac{p_1^2}{\epsilon}$ or $p_1|q_i|<p_i$. Notice that the first restriction is not a topological restriction.

If $A_1=\frac{1}{2\nu}(p_1-\sqrt{p_1^2+\nu\epsilon q_1^2})$ and $A_i=\frac{1}{2\nu}(p_i+\sqrt{p_i^2+\nu\epsilon q_i^2})$, then \eqref{c3} always fails.

If $A_1=\frac{1}{2\nu}(p_1-\sqrt{p_1^2+\nu\epsilon q_1^2})$ and $A_i=\frac{1}{2\nu}(p_i-\sqrt{p_i^2+\nu\epsilon q_i^2})$, then the resulting topological restriction is $p_1|q_i|<p_i$.

Now it is time to summarize what we have obtained above.
\begin{proposition}\label{fes}
Let $(\hat{M}^m,\hat{g})$ be as in Proposition \ref{loc} with $\kappa_0=0$, and $s=0$ corresponding to $t=0$. We can smoothly compactify $(\hat{M},\hat{g})$ at $s_\ast=-\frac{1}{2A_1}$, and then extend it to $s=0$ provided $\alpha$ and $\beta_i$, $1\le i\le r$, are given by \eqref{cal} and \eqref{cbe}, where $w=\prod_{i=1}^r\beta_i^{n_i}$, $m=2+\sum_{i=1}^r2n_i$, $p=m+n$, $p_1=n_1+1$ and $|q_1|=1$. The involved constants $\nu$ and $A_i$ are determined as follows:

if $\epsilon=0$, and $p_1|q_i|>p_i$ for all $2\le i\le r$, then $\nu<0$, and $A_i=\frac{p_i}{\nu}$, $1\le i\le r$;

if $\epsilon<0$, and $p_1|q_i|>p_i$ for all $2\le i\le r$, then $\nu<0$, and $A_i=\frac{1}{2\nu}(p_i+\sqrt{p_i^2+\nu\epsilon q_i^2})$, $1\le i\le r$;

if $\epsilon>0$, define $\phi=\min_{1\le i\le r}\{p_i^2q_i^{-2}\}\le p_1^2$. There are several cases.
\begin{enumerate}
 \item if $p_1|q_i|>p_i$ for some $2\le i\le r$, then $\phi<p_1^2$. We can choose $\nu\in[-\frac{\phi}{\epsilon},0)$, $A_1=\frac{1}{2\nu}(p_1+\sqrt{p_1^2+\nu\epsilon q_1^2})$ and $A_i=\frac{1}{2\nu}(p_i+\delta_i\sqrt{p_i^2+\nu\epsilon q_i^2})$, $2\le i\le r$, with $\delta_i=\pm1$ if $p_1|q_i|>p_i$, or $\delta_i=-1$ if $p_1|q_i|\le p_i$;
 \item if $p_1|q_i|\le p_i$ for all $2\le i\le r$, then $\phi=p_1^2$. There are two subcases.
\begin{enumerate}
 \item if $p_1|q_i|=p_i$ for some $2\le i\le r$, we can choose $\nu\in(-\frac{\phi}{\epsilon},0)$, $A_1=\frac{1}{2\nu}(p_1+\sqrt{p_1^2+\nu\epsilon q_1^2})$ and $A_i=\frac{1}{2\nu}(p_i-\sqrt{p_i^2+\nu\epsilon q_i^2})$, $2\le i\le r$;
 \item if $p_1|q_i|<p_i$ for all $2\le i\le r$, we can choose $\nu\in[-\frac{\phi}{\epsilon},0)$, $A_1=\frac{1}{2\nu}(p_1\pm\sqrt{p_1^2+\nu\epsilon q_1^2})$ and $A_i=\frac{1}{2\nu}(p_i-\sqrt{p_i^2+\nu\epsilon q_i^2})$, $2\le i\le r$.
\end{enumerate}
\end{enumerate}
\end{proposition}

\subsection{Geometric discussion}\label{geo}

We now move on to the geometric implications of those solutions in Proposition \ref{fes}. After smoothly compactifying and extending $(\hat{M}^m,\hat{g})$ described above, we obtain a compact manifold $(\bar{M}^m,g_{\bar{M}})$ with interior $M$ and boundary $\partial M$, which has $(\hat{M},\hat{g})$ as an open and dense subset. Topologically, $\partial M=\{0\}\times P_q$ is a copy of the principal circle bundle over the product $\mathbb{C}P^{n_1}\times V_2\times\cdots\times V_r$, and $M$ is the total space of the corresponding complex $B^{2n_1+2}$-bundle over $V_2\times\cdots\times V_r$. Since $\rho=\kappa_1s$ satisfies $\rho>0$ in $M$, $\rho=0$ and $d\rho=\kappa_1ds\ne0$ on $\partial M$, it serves as a defining function for $\partial M$. Let $(N^n,g_N)$, $n=p-m\ge0$, be an arbitrary connected closed Einstein manifold with Einstein constant $\epsilon$.
\begin{lemma}\label{comei}
$(M^m\times N^n,\rho^{-2}(g_M+g_N))$ is a complete Einstein manifold.
\end{lemma}
\begin{proof}
Since \eqref{ce} holds true for the triple $(M,g_M,\rho)$ with $g_M=g_{\bar{M}}|_M$, it follows from Proposition \ref{eif} that $(M\times N,\rho^{-2}(g_M+g_N))$ is an Einstein manifold. Notice that
\begin{align}\label{gee}
\rho^{-2}(g_M+g_N)=\frac{ds^2}{\rho(s)^2\alpha(s)}+\frac{\alpha(s)}{\rho(s)^2}\theta\otimes\theta
+\sum_{i=1}^r\frac{\beta_i(s)}{\rho(s)^2}\hat{\pi}^\ast h_i+\rho(s)^{-2}g_N,
\end{align}
with $s\in(0,s_\ast]$. The geodesic distance given by
\begin{align*}
\int_0^{s_\ast}\frac{d\tau}{\rho(\tau)\sqrt{\alpha(\tau)}}
\end{align*}
is infinite because of $\rho(s)\sqrt{\alpha(s)}=O(s)$ as $s\searrow0$. We therefore conclude that the Einstein manifold $(M\times N,\rho^{-2}(g_M+g_N))$ is complete.
\end{proof}

In what follows, we shall normalize the Einstein metric $g=\rho^{-2}(g_M+g_N)$ such that its Einstein constant becomes $1-p$. By Proposition \ref{eif}, this is equivalent to set
\begin{align*}
\epsilon\rho^2-\rho\triangle_M\rho-(p-1)|d\rho|_M^2=1-p.
\end{align*}
We proceed to compute $\triangle_M\rho$ and $|d\rho|_M^2$.
\begin{lemma}
\begin{align}\label{lap}
-\triangle_M\rho=\kappa_1((p-1)\frac{\alpha}{s}-\epsilon s+\frac{\nu}{s}).
\end{align}
\end{lemma}
\begin{proof}
Since $ds=fdt$, we have $\frac{d}{dt}=\frac{ds}{dt}\frac{d}{ds}=f\frac{d}{ds}$, and hence
\begin{eqnarray*}
&&\dot{f}=ff'=\frac{\alpha'}{2},\\
&&\dot{g_i}=fg_i'=\frac{f}{2g_i}\beta_i',\\
&&\dot{\rho}=f\rho'=\kappa_1 f,\\
&&\ddot{\rho}=f(\kappa_1 f)'=\frac{\kappa_1}{2}\alpha',
\end{eqnarray*}
where $\dot{}$ and $'$ denote respectively the differentiation with respect to $t$ and $s$. It thus follows from Lemma \ref{hess} that
\begin{eqnarray*}
-\triangle_M\rho&=&\ddot{\rho}+\frac{\dot{f}}{f}\dot{\rho}+\sum_{i=1}^r2n_i\frac{\dot{g_i}}{g_i}\dot{\rho}\\
&=&\kappa_1(\alpha'+\alpha\sum_{i=1}^rn_i\frac{\beta_i'}{\beta_i})\\
&=&\kappa_1(\alpha'+\alpha(\log w)')\\
&=&\kappa_1((p-1)\frac{\alpha}{s}-\epsilon s+\frac{\nu}{s}),
\end{eqnarray*}
where the last equality comes from \eqref{sie}.
\end{proof}
\begin{lemma}
\begin{align}\label{gra}
|d\rho|_M^2=\kappa_1^2\alpha.
\end{align}
\end{lemma}
\begin{proof}
Actually, $|d\rho|_M^2=\dot{\rho}^2=f^2(\rho')^2=\kappa_1^2\alpha$.
\end{proof}
\begin{proposition}
If $\kappa_1=\sqrt{\frac{1-p}{\nu}}$, then the Einstein metric $g$ has Einstein constant $1-p$.
\end{proposition}
\begin{proof}
Combining \eqref{lap} and \eqref{gra} leads to
\begin{align*}
\epsilon\rho^2-\rho\triangle_M\rho-(p-1)|d\rho|_M^2=\epsilon\kappa_1^2s^2+\kappa_1^2s((p-1)\frac{\alpha}{s}-\epsilon s+\frac{\nu}{s})-(p-1)\kappa_1^2\alpha=\kappa_1^2\nu.
\end{align*}
The normalization requires that $\kappa_1^2\nu=1-p$, i.e., $\kappa_1=\sqrt{\frac{1-p}{\nu}}>0$.
\end{proof}

\begin{proof}[Proof of Theorems \ref{positive} and \ref{zero}] Lemma \ref{comei} says that $(M\times N,g)$ is a complete Einstein manifold. Furthermore, the conformal deformation $(M\times N,g_M+g_N)$ has a smooth compact extension as $(\bar{M}\times N,g_{\bar{M}}+g_N)$. By definition, $(M\times N,g)$ is a conformally compact Einstein manifold. It turns out that there remains a unique free parameter $\nu$ after we take into account homothety. This completes the proof of Theorems \ref{positive} and \ref{zero}.
\end{proof}

\begin{remark}
The conformal infinity of $(M\times N,g)$ is $(\partial M\times N,[\gamma])$ with conformal metric
\begin{eqnarray}
\notag\gamma&=&\alpha(0)^{-1}(g_{\bar{M}}+g_N)|_{T(\partial M\times N)}\\
\notag&=&\theta\otimes\theta+\sum_{i=1}^r\frac{\beta_i(0)}{\alpha(0)}\hat{\pi}^\ast h_i+\alpha(0)^{-1}g_N\\
\label{conf}&=&\theta\otimes\theta+\sum_{i=1}^r\frac{(p-1)q_i^2}{4\nu A_i}\hat{\pi}^\ast h_i+\frac{1-p}{\nu}g_N
\end{eqnarray}
as $\alpha(0)=\frac{\nu}{1-p}$ and $\beta_i(0)=-\frac{q_i^2}{4A_i}$, $1\le i\le r$. In particular, $\gamma$ has constant scalar curvature.
\end{remark}

\subsection{non-Fano K\"{a}hler-Einstein bases}\label{nfk}

In the preceding constructions, it is possible to allow some of the base manifolds $(V_i,J_i,h_i)$, $2\le i\le r$, to be closed K\"{a}hler-Einstein manifolds with non-positive Einstein constants.

If, say, $(V_2,h_2)$ is a K\"{a}hler-Einstein manifold with negative Einstein constant, then we can assume the first Chern class $c_1(V_2)=p_2a_2$, where $p_2\in\mathbb{Z}_-$ is a negative integer, and $a_2\in H^2(V_2;\mathbb{Z})$ is an indivisible class. We shall normalize the K\"{a}hler-Einstein metric $h_2$ such that $\mbox{Ric}(h_2)=p_2h_2$, and hence the K\"{a}hler class $\eta_2=2\pi[a_2]$. We go through the previous constructions, and arrive at the consistency condition (cf. \eqref{csc})
\begin{align}\label{nfn}
\frac{\epsilon q_2^2}{4A_2^2}+\frac{p_2}{A_2}=\nu.
\end{align}
This forces $\epsilon$ to be negative since $p_2$, $A_2$ and $\nu$ are all negative. We can then rewrite \eqref{nfn} as a quadratic equation for $A_2$ as follows
\begin{align}\label{nfns}
A_2^2-\frac{p_2}{\nu}A_2-\frac{\epsilon q_2^2}{4\nu}=0.
\end{align}
The discriminant of \eqref{nfns} is $\Delta_2=(\frac{p_2}{\nu})^2+\frac{\epsilon q_2^2}{\nu}>0$. Hence \eqref{nfns} always has real roots. Furthermore, \eqref{nfns} has a unique negative root
\begin{align*}
A_2=\frac{1}{2}(\frac{p_2}{\nu}-\sqrt{\Delta_2})=\frac{1}{2\nu}(p_2+\sqrt{p_2^2+\nu\epsilon q_2^2}).
\end{align*}
The inequality \eqref{c3} is satisfied without any restriction in this case.

If, say, $(V_2,h_2)$ is a Ricci-flat K\"{a}hler manifold, i.e., $c_1(V_2)=0$, and hence $p_2=0$, then we need to assume the K\"{a}hler class $\eta_2=2\pi[a_2]$, where $a_2\in H^2(V_2;\mathbb{Z})$ is also an indivisible class, i.e., $(V_2,h_2)$ is a Hodge manifold. We go through again the above constructions, and arrive at the consistency condition
\begin{align}\label{nff}
\frac{\epsilon q_2^2}{4A_2^2}=\nu.
\end{align}
Thus $\epsilon<0$ as $\nu<0$. Clearly, \eqref{nff} has a unique negative root
\begin{align*}
A_2=-\sqrt{\frac{\epsilon q_2^2}{4\nu}}.
\end{align*}
The inequality \eqref{c3} is also satisfied without any restriction in this case.

We therefore have the following analogue of Proposition \ref{fes}.
\begin{proposition}\label{fesnp}
Let $\hat{M}^m=(0,s_\ast)\times P_q$, where $P_q$ is the principal circle bundle as in Theorem \ref{np} with principal connection $\theta$, and
\begin{align*}
\hat{g}=\alpha(s)^{-1}ds^2+\alpha(s)\theta\otimes\theta+\sum_{i=1}^r\beta_i(s)\hat{\pi}^\ast h_i.
\end{align*}
If $\epsilon<0$ and $p_1|q_i|>p_i$, $2\le i\le r$, then we can smoothly compactify $(\hat{M},\hat{g})$ at $s_\ast=-\frac{1}{2A_1}$, and extend it to $s=0$ provided $\alpha$ and $\beta_i$ are given by \eqref{cal} and \eqref{cbe} with $w=\prod_{i=1}^r\beta_i^{n_i}$, $m=2+\sum_{i=1}^r2n_i$, $p=m+n$, $p_1=n_1+1$, $|q_1|=1$, $\nu<0$ and $A_i=\frac{1}{2\nu}(p_i+\sqrt{p_i^2+\nu\epsilon q_i^2})$, $1\le i\le r$.
\end{proposition}
Theorem \ref{np} now follows from the same arguments as in \S\ref{geo}.

\subsection{Examples}\label{ex}

In the rest of this section, we shall give some concrete special cases of the conformally compact Einstein manifolds constructed in Theorems \ref{positive}, \ref{zero} and \ref{np}. For simplicity, we specialize to the situation where $N$ is a point or a circle, and the base of the principal circle bundle is the product of a point with a complex projective space of positive dimensions.

In the first two examples, let $V_1=\mathbb{C}P^{n_1}$, $n_1>0$, and $V_2$ be a point. Thus $M=B^{2n_1+2}$. Since $|q_1|=1$, the total space $P_{\pm1}$ is diffeomorphic to $S^{2n_1+1}$, and $\hat{\pi}:S^{2n_1+1}\rightarrow\mathbb{C}P^{n_1}$ is the Hopf fibration. In the last two examples, let $V_1$ be a point, and $V_2=\mathbb{C}P^{n_2}$, $n_2>0$. Thus $M^{2n_2+2}$ is the total space of a nontrivial complex one-disc bundle over $\mathbb{C}P^{n_2}$.

\subsubsection{$n=0$, $n_1>0$, $n_2=0$}\label{hyp}

In this case, $N$ is a point. So is the event horizon. Since $p=2+2n_1=2p_1$, the conformal infinities are represented by $(2n_1+1)$-dimensional squashed spheres with conformal metrics (cf. \eqref{conf})
\begin{align*}
\gamma=\theta\otimes\theta+\frac{2p_1-1}{4\nu A_1}\hat{\pi}^\ast h_1.
\end{align*}
This family of conformally compact Einstein metrics generalize to higher dimensions the $4$-dimensional AdS-Taub-NUT metrics (corresponding to $n_1=1$) \cite{HawHunPag99, Ped86}.

By Proposition \ref{fes}, $A_1=\frac{1}{2\nu}(p_1+\sqrt{p_1^2+\nu\epsilon})$ with $\nu\epsilon\ge-p_1^2$, or $A_1=\frac{1}{2\nu}(p_1-\sqrt{p_1^2+\nu\epsilon})$ with $-p_1^2\le\nu\epsilon<0$. In the former case, when $\nu\epsilon=1-2p_1$ and $A_1=\frac{1}{2\nu}(2p_1-1)$,
\begin{align*}
\gamma=\theta\otimes\theta+\frac{1}{2}\hat{\pi}^\ast h_1
\end{align*}
is the standard metric on $S^{2n_1+1}(1)$ of constant curvature $1$. A further computation then shows $\alpha=2\beta_1$. Thus the corresponding conformally compact Einstein metric is (cf. \eqref{gee})
\begin{align*}
g=\frac{ds^2}{\rho(s)^2\alpha(s)}+\frac{\alpha(s)}{\rho(s)^2}(\theta\otimes\theta+\frac{1}{2}\hat{\pi}^\ast h_1),
\end{align*}
which leads to hyperbolic space $H^{2n_1+2}$ of constant curvature $-1$.

\subsubsection{$n=1$, $n_1>0$, $n_2=0$}\label{quo}

In this case, $N$ is a circle. So is the event horizon. Since $p=3+2n_1=1+2p_1$, $\epsilon=0$ and $A_1=\frac{p_1}{\nu}$, the conformal infinity is represented by the conformal metric
\begin{align*}
\gamma=\theta\otimes\theta+\frac{1}{2}\hat{\pi}^\ast h_1-\frac{2p_1}{\nu}g_{S^1},
\end{align*}
which is the product metric on $S^{2n_1+1}(1)\times S^1$. A further computation then shows $\alpha=2\beta_1$ as well. Thus the corresponding conformally compact Einstein metric is
\begin{align*}
g=\frac{ds^2}{\rho(s)^2\alpha(s)}+\frac{\alpha(s)}{\rho(s)^2}(\theta\otimes\theta+\frac{1}{2}\hat{\pi}^\ast h_1)+\rho(s)^{-2}g_{S^1},
\end{align*}
which leads to the quotient of a certain portion of the Euclidean AdS space \cite{HawPag82, Wit98}.

\subsubsection{$n=0$, $n_1=0$, $n_2>0$}

In this case, $N$ is a point, and the event horizon is $\mathbb{C}P^{n_2}$. Since $p=2+2n_2=2p_2$, the conformal infinities are represented by $(2n_2+1)$-dimensional squashed spheres with conformal metrics
\begin{align*}
\gamma=\theta\otimes\theta+\frac{(2p_2-1)q_2^2}{4\nu A_2}\hat{\pi}^\ast h_2.
\end{align*}
This family of conformally compact Einstein metrics generalize to higher dimensions the $4$-dimensional AdS-Taub-Bolt metrics (corresponding to $n_2=1$) \cite{HawHunPag99}.

Since $\phi=\min\{1,p_2^2q_2^{-2}\}$ as $p_1=|q_1|=1$, there are three distinct cases.

If $|q_2|>p_2$, then $\phi=p_2^2q_2^{-2}<1$, and either $A_i=\frac{1}{2\nu}(p_i+\sqrt{p_i^2+\nu\epsilon q_i^2})$, $1\le i\le2$, with $\nu\epsilon\ge-p_2^2q_2^{-2}$, or $A_1=\frac{1}{2\nu}(p_1+\sqrt{p_1^2+\nu\epsilon q_1^2})$ and $A_2=\frac{1}{2\nu}(p_2-\sqrt{p_2^2+\nu\epsilon q_2^2})$, with $-p_2^2q_2^{-2}\le\nu\epsilon<0$. In the former case, when $\nu\epsilon=(1-2p_2)q_2^{-2}$ and $A_2=\frac{1}{2\nu}(2p_2-1)$,
\begin{align*}
q_2^{-2}\gamma=q_2^{-2}\theta\otimes\theta+\frac{1}{2}\hat{\pi}^\ast h_2
\end{align*}
is the induced metric on the quotient $L^{n_2}_{q_2}$ of $S^{2n_2+1}(1)$ by a cyclic group of order $|q_2|$.

If $|q_2|=p_2$, then $\phi=1$, $A_1=\frac{1}{2\nu}(p_1+\sqrt{p_1^2+\nu\epsilon q_1^2})$ and $A_2=\frac{1}{2\nu}(p_2-\sqrt{p_2^2+\nu\epsilon q_2^2})$, with $-1<\nu\epsilon<0$.

If $|q_2|<p_2$, then $\phi=1$, $A_1=\frac{1}{2\nu}(p_1\pm\sqrt{p_1^2+\nu\epsilon q_1^2})$ and $A_2=\frac{1}{2\nu}(p_2-\sqrt{p_2^2+\nu\epsilon q_2^2})$, with $-1\le\nu\epsilon<0$. For each pair of parameters $(\epsilon,\nu)$, there exist two non-isometric conformally compact Einstein metrics with the same conformal infinity. Such a non-uniqueness phenomena has been observed first by Hawking and Page \cite{HawPag82}. More examples of non-uniqueness can be found in \cite{And05}.

\subsubsection{$n=1$, $n_1=0$, $n_2>0$}\label{cir}

In this case, $N$ is a circle, and the event horizon is $\mathbb{C}P^{n_2}\times S^1$. Since $p=3+2n_2=1+2p_2$, $\epsilon=0$ and $A_2=\frac{p_2}{\nu}$, the conformal infinity is represented by the conformal metric
\begin{align*}
q_2^{-2}\gamma=q_2^{-2}\theta\otimes\theta+\frac{1}{2}\hat{\pi}^\ast h_2-\frac{2p_2}{\nu q_2^2}g_{S^1}
\end{align*}
with $|q_2|>p_2$, which is the product metric on $L^{n_2}_{q_2}\times S^1$. We should mention that there is another conformally compact Einstein manifold $K^{n_2}_{q_2}$ (cf., for instance, \cite[Example 2.1]{And05}), which also has $(L^{n_2}_{q_2}\times S^1, [q_2^{-2}\gamma])$ as the conformal infinity. The relation of our example and $K^{n_2}_{q_2}$ is reminiscent of that of Example \ref{quo} and the Euclidean AdS Schwarzschild black hole discussed in \cite[\S3]{Wit98} (cf. also \cite{HawPag82}).
\begin{remark}
If we replace $\mathbb{C}P^{n_2}$ by a general Fano K\"{a}hler-Einstein manifold $(V_2^{n_2},h_2)$ of positive dimensions, then the conformal infinity is represented by the product $P_{q_2}\times S^1$, where $P_{q_2}$ is a principal circle bundle over $V_2$, with conformal metric
\begin{eqnarray*}
\notag\gamma=\theta\otimes\theta+\frac{(n_2+1)q_2^2}{2p_2}\hat{\pi}^\ast h_2-\frac{2(n_2+1)}{\nu}g_{S^1}.
\end{eqnarray*}
Notice that $\theta\otimes\theta+\frac{(n_2+1)q_2^2}{2p_2}\hat{\pi}^\ast h_2$ is a positive Einstein metric on $P_{q_2}$ (cf. Lemma \ref{nv}). Thus there is again another conformally compact Einstein manifold (cf. \cite[Example 2.1]{And05}), which also has $(P_{q_2}\times S^1,[\gamma])$ as the conformal infinity.
\end{remark}

\section{Conformal Invariants}\label{ci}

In this section, we shall compute the conformal invariants, i.e., the renormalized volume in even dimensions and the conformal anomaly in odd dimensions, associated to the conformally compact Einstein manifolds $(M^m\times N^n,\rho^{-2}(g_M+g_N))$ constructed in \S\ref{constr}.

\subsection{Geodesic defining functions}\label{gdf}

Recall that the conformally compact Einstein metric $g=\rho^{-2}(g_M+g_N)$ constructed in \S\ref{constr} is of the form given by \eqref{gee}. It turns out that $\rho$ is not a geodesic defining function (cf. \eqref{gc}). In order to find a geodesic defining function $\sigma$, we are led to a first-order ordinary differential equation
\begin{align}\label{fd}
\frac{d\sigma}{\sigma}=\frac{ds}{\rho\sqrt{\alpha}}.
\end{align}
\begin{proposition}\label{ggdf}
Let
\begin{eqnarray*}
\zeta(s)=\int_0^s(\frac{1}{\kappa_1\sqrt{\alpha(\tau)}}-1)\frac{d\tau}{\tau},&&0\le s\le s_\ast.
\end{eqnarray*}
Then $\sigma=s\exp\zeta$ satisfies \eqref{fd}, and is a global geodesic defining function.
\end{proposition}
\begin{proof}
First of all, we need to show that $\zeta$ as above is well-defined on $[0,s_\ast]$. Notice that $\kappa_1=\sqrt{\frac{1-p}{\nu}}=\frac{1}{\sqrt{\alpha(0)}}$. It follows from l'H\^{o}pital's rule that (cf. Remark \ref{ap})
\begin{align*}
\lim_{s\searrow0}(\frac{1}{\kappa_1\sqrt{\alpha(s)}}-1)/s=-\frac{\alpha'(0)}{2\alpha(0)}=0.
\end{align*}
Thus $\zeta(0)=0$, and $\zeta$ is well-defined in $(0,s_\ast)$. Around $s_\ast$ where $\alpha(s_\ast)=0$ and $\alpha'(s_\ast)=-2$ (cf. Proposition \ref{alc}), we have
\begin{align*}
\alpha(s)=-2(s-s_\ast)+o(s-s_\ast)=O(s-s_\ast).
\end{align*}
Hence $\zeta$ has a finite limit when $s$ increases to $s_\ast$, i.e., $\zeta$ is also well-defined at $s_\ast$.

It is clear that
\begin{eqnarray}\label{zet}
\zeta'=\frac{1}{s}(\frac{1}{\kappa_1\sqrt{\alpha}}-1),&&0<s<s_\ast,
\end{eqnarray}
from which we see that $\zeta$ is smooth on $[0,s_\ast)$. It remains to check the smoothness of $\zeta$ at $s_\ast$. To do this, we return to the coordinate $t$. By \eqref{zet}, we have
\begin{align*}
\frac{d\zeta}{dt}=\frac{d\zeta}{ds}\frac{ds}{dt}=\frac{1}{s}(\frac{1}{\kappa_1}-\sqrt{\alpha})=\frac{1}{s(t)}(\frac{1}{\kappa_1}-f(t)),
\end{align*}
from which we see that $\zeta$ is smooth at $t_\ast$ as well.

Finally, it is easy to check that $\sigma=s\exp\zeta$ satisfies \eqref{fd}, and is thus a global geodesic defining function. This completes the proof of Proposition \ref{ggdf}.
\end{proof}
It follows from \eqref{fd} that
\begin{eqnarray*}
\frac{d\sigma}{ds}=\frac{\sigma}{\rho\sqrt{\alpha}}>0,&&0<s<s_\ast.
\end{eqnarray*}
The inverse function theorem asserts that $s$ can be written as a function of $\sigma$ on $(0,\sigma_\ast)$, with $\sigma_\ast=\sigma(s_\ast)$, say $s=\sigma\xi$, where $\xi=\xi(\sigma)$ is a smooth function of $\sigma$ such that $\xi(0)=1$ as $\zeta(0)=0$. It turns out that $\xi$ satisfies the ordinary differential equation
\begin{align}\label{cp}
\sigma\frac{d\xi}{d\sigma}=\xi(\kappa_1\sqrt{\alpha}-1).
\end{align}

\subsection{Asymptotic volumes}\label{aav}

In terms of the geodesic defining function $\sigma$, we can rewrite our Einstein metric as
\begin{align}\label{gm}
g=\frac{d\sigma^2}{\sigma^2}+E(\sigma)^2\theta\otimes\theta+\sum_{i=1}^rF_i(\sigma)^2\hat{\pi}^\ast h_i+G(\sigma)^2g_N
\end{align}
with $E(\sigma)=\rho(s)^{-1}\sqrt{\alpha(s)}$, $F_i(\sigma)=\rho(s)^{-1}\sqrt{\beta_i(s)}$ and $G(\sigma)=\rho(s)^{-1}$. By \eqref{fd} and \eqref{cp}, the corresponding asymptotic volume is
\begin{eqnarray*}
\mbox{Vol}_g(\{\sigma>\delta\})&=&C_3\int_\delta^{\sigma_\ast}E\prod_{i=1}^rF_i^{2n_i}G^n\frac{d\sigma}{\sigma}\\
&=&C_3\int_{s(\delta)}^{s_\ast}\rho^{-1}\sqrt{\alpha}\prod_{i=1}^r(\rho^{-1}\sqrt{\beta_i})^{2n_i}
\rho^{-n}\frac{ds}{\rho\sqrt{\alpha}}\\
&=&C_3\int_{s(\delta)}^{s_\ast}\frac{w}{\rho^p}ds\\
&=&\frac{C_1C_3}{\kappa_1^p}\sum_{j=0}^{m/2-1}a_j\int_{s(\delta)}^{s_\ast} s^{2j-p}ds,
\end{eqnarray*}
where $C_3=2\pi\mbox{Vol}_{g_N}(N)\prod_{i=1}^r\mbox{Vol}_{h_i}(V_i)$, $C_1$ and $a_j$'s are given in the proof of Proposition \ref{ral} (cf. \eqref{aj} for example). Since
\begin{eqnarray*}
2j-p+1\le2(\frac{m}{2}-1)-(m+n)+1=-1-n<0,&&0\le j\le\frac{m}{2}-1,
\end{eqnarray*}
no $\log s$ term appears after integration. Therefore,
\begin{eqnarray}\label{asy}
\notag\mbox{Vol}_g(\{\sigma>\delta\})
&=&\frac{C_1C_3}{\kappa_1^p}\sum_{j=0}^{m/2-1}\frac{a_j}{2j-p+1}s^{2j-p+1}|_{s(\delta)}^{s_\ast}\\
\label{av}&=&\frac{C_1C_3}{\kappa_1^p}\sum_{j=0}^{m/2-1}\frac{a_j}{2j-p+1}(s_\ast^{2j-p+1}-\delta^{2j-p+1}\xi(\delta)^{2j-p+1}).
\end{eqnarray}

\subsection{Conformal anomalies}\label{ca}

Since $\xi(\delta)$ is a smooth function of $\delta$, no $\log\delta$ term appears in the asymptotic expansion \eqref{av} of $\mbox{Vol}_g(\{r>\delta\})$. Hence, when $p$ is odd, the conformal anomaly $L$ is always zero.

Recall that $L$ equals a constant multiple of the integral of $Q$-curvature on the conformal infinity (cf. \eqref{confa}). One may thus ask whether some representative metric in the conformal infinity indeed has vanishing $Q$-curvature. This is answered by Theorem \ref{vca}.

\begin{proof}[Proof of Theorem \ref{vca}]
The theorem clearly follows if there is a representative metric with constant $Q$-curvature.

Consider the restriction $g_b$ of $\sigma^2g$ to the boundary $\partial M\times N$. Since
\begin{align*}
\lim_{s\searrow0}\frac{\sigma^2}{\rho(s)^2}=\lim_{s\searrow0}\kappa_1^{-2}(\exp\zeta(s))^2=\kappa_1^{-2},
\end{align*}
it follows from \eqref{gm} that
\begin{align}\label{can}
g_b=\kappa_1^{-2}(\alpha(0)\theta\otimes\theta+\sum_{i=1}^r\beta_i(0)\hat{\pi}^\ast h_i+g_N).
\end{align}

We invoke now Theorem 3.1 in \cite{FefGra02}, which suggests that we find the unique solution $\tilde{U}$ mod $O(\sigma^{p-1})$ of the Poisson equation
\begin{align*}
\triangle_g\tilde{U}=p-1+O(\sigma^p\log\sigma)
\end{align*}
of the form
\begin{align*}
\tilde{U}=\log\sigma +A+B\sigma^{p-1}\log\sigma+O(\sigma^{p-1}),
\end{align*}
with $A,B\in C^\infty(M\times N)$ and $A|_{\partial M\times N}=0$. Then $B|_{\partial M\times N}$ will be a constant multiple of the $Q$-curvature of $g_b$.

For convenience, we define $\iota=\log\sigma$, i.e., $\sigma=\exp\iota$, and write our Einstein metric as
\begin{align*}
g=d\iota^2+\bar{E}(\iota)^2\theta\otimes\theta+\sum_{i=1}^r\bar{F}_i(\iota)^2\pi^\ast h_i+\bar{G}(\iota)^2g_N,
\end{align*}
with $\bar{E}(\iota)=E(\sigma)$, $\bar{F}_i(\iota)=F_i(\sigma)$ and $\bar{G}(\iota)=G(\sigma)$. It thus follows from Lemma \ref{hess} that
\begin{align*}
\triangle_g\iota=\frac{1}{\bar{E}}\frac{d\bar{E}}{d\iota}+\sum_{i=1}^r\frac{2n_i}{\bar{F}_i}\frac{d\bar{F}_i}{d\iota}
+\frac{n}{\bar{G}}\frac{d\bar{G}}{d\iota}
=\frac{d}{d\iota}\log(\bar{E}\prod_{i=1}^r\bar{F}_i^{2n_i}\bar{G}^n),
\end{align*}
which implies that $\triangle_g\iota$ depends on $\iota$ only. In other words, $\triangle_g\log\sigma$ is a function of $\sigma$ only.

As in the proof of Theorem 3.1 in \cite{FefGra02}, we shall first construct $A$ mod $O(\sigma^{p-1})$ of the form $\sum_{i=2}^{p-2}c_i\sigma^i$ with constants $c_i$ such that
\begin{align*}
\triangle_gA=p-1-\triangle_g\log\sigma+\sigma^{p-1}\tilde{W},
\end{align*}
where the remainder $\tilde{W}$ turns out to be a function of $\sigma$ only. Secondly, we can determine formally $B=B(\sigma)$, which satisfies the boundary condition $B|_{\partial M\times N}=-(p-1)^{-1}\tilde{W}(0)$. It then follows from this boundary condition that the $Q$-curvature of $g_b$ is a constant.
\end{proof}

\begin{remark}
Recall that the $Q$-curvature in dimension $4$ is given by a simple formula \eqref{qcur}. It is therefore easy to check that the representative metrics $g_b$ (cf. \eqref{can}) in the conformal infinities of the $5$-dimensional conformally compact Einstein manifolds from Examples \ref{quo} and \ref{cir} have indeed vanishing $Q$-curvature (cf. Lemma \ref{nv}).
\end{remark}

\begin{proof}[Proof of Corollary \ref{rpzq}]
The first part of Corollary \ref{rpzq} is an immediate consequence of Theorem \ref{vca}. Hence we need only to prove the second part of the corollary, i.e., the one-parameter family of squashed spheres contain the standard sphere of constant curvature $1$ provided $(N^n,g_N)$ has Einstein constant $-(n-1)$.

We now specialize to the case $n_1>0$, $n_2=0$ and $r=2$. Since $|q_1|=1$, $p_1=n_1+1$ and $p=n+2n_1+2=n+2p_1$, where $n$ is an odd positive integer, the conformal infinities are represented by $S^{2n_1+1}\times N$ with conformal metrics (cf. \eqref{conf})
\begin{align*}
\gamma=\theta\otimes\theta+\frac{n+2p_1-1}{4\nu A_1}\hat{\pi}^\ast h_1+\frac{1-n-2p_1}{\nu}g_N.
\end{align*}
We shall divide our discussion into two cases $\epsilon=0$ and $\epsilon<0$.

When $\epsilon=0$, it follows from Proposition \ref{fes} that $\nu<0$ and $A_1=\frac{p_1}{\nu}$. Thus
\begin{align*}
\frac{n+2p_1-1}{4\nu A_1}=\frac{1}{2}+\frac{n-1}{4p_1}\ge\frac{1}{2}
\end{align*}
with equality iff $n=1$, i.e., $N=S^1$. In this case, if we choose $\nu=-2p_1$, then
\begin{align*}
\gamma=\theta\otimes\theta+\frac{1}{2}\hat{\pi}^\ast h_1+g_{S^1}
\end{align*}
is the product metric on $S^{2n_1+1}(1)\times S^1$, which has zero $Q$-curvature by Theorem \ref{vca}.

When $\epsilon<0$, and hence $n>1$, we have $\nu<0$ and $A_1=\frac{1}{2\nu}(p_1+\sqrt{p_1^2+\nu\epsilon})$. Thus
\begin{align*}
\frac{n+2p_1-1}{4\nu A_1}=\frac{n+2p_1-1}{2(p_1+\sqrt{p_1^2+\nu\epsilon})},
\end{align*}
which is equal to $\frac{1}{2}$ iff
\begin{align*}
\nu=\frac{(n-1)(2p_1+n-1)}{\epsilon}.
\end{align*}
In this case, if we normalize the Einstein metric $g_N$ such that $\epsilon=-(n-1)$, then
\begin{align*}
\gamma=\theta\otimes\theta+\frac{1}{2}\hat{\pi}^\ast h_1+g_{N}
\end{align*}
is the product metric on $S^{2n_1+1}(1)\times N$, which also has zero $Q$-curvature by Theorem \ref{vca}.
\end{proof}
\begin{remark}
When $\epsilon>0$, and therefore $n>1$, it follows from Proposition \ref{fes} that $\nu<0$ and $A_1=\frac{1}{2\nu}(p_1\pm\sqrt{p_1^2+\nu\epsilon})>\frac{p_1}{\nu}$, i.e., $0<\nu A_1<p_1$. Thus
\begin{align*}
\frac{n+2p_1-1}{4\nu A_1}>\frac{n+2p_1-1}{4p_1}>\frac{1}{2}.
\end{align*}
In this case, $\gamma$ cannot be the product metric on $S^{2n_1+1}(1)\times N$.
\end{remark}

\subsection{Renormalized volumes}\label{rv}

In the rest of this section, we shall assume $p$ is even, and compute the renormalized volume, i.e., the constant term in the asymptotic expansion \eqref{rv} of $\mbox{Vol}_g(\{\sigma>\delta\})$. More precisely, we will determine the coefficient of $\delta^{p-1-2j}$ term in the Maclaurin series of $\xi(\delta)^{2j-p+1}$, i.e., $\frac{1}{(p-1-2j)!}\frac{d^{p-1-2j}}{d\sigma^{p-1-2j}}(\xi(\sigma)^{2j-p+1})(0)$, $0\le j\le\frac{m}{2}-1$. We proceed to establish several lemmas.
\begin{lemma}\label{odav}
$\alpha(s)$ is even at $s=0$ up to order $p-2$, i.e., $\frac{d^{2j+1}\alpha}{ds^{2j+1}}(0)=0$, $j=0,\cdots,\frac{p}{2}-2$.
\end{lemma}
\begin{proof}
By \eqref{alpq}, we have
\begin{align}\label{aqp}
\alpha Q=P,
\end{align}
where
\begin{eqnarray*}
&&P=\frac{b_0}{p-1}+\frac{b_1}{p-3}s^2+\cdots+\frac{b_{m/2}}{n-1}s^m+C_2s^{p-1},\\
&&Q=a_0+a_1s^2+\cdots+a_{m/2-1}s^{m-2}
\end{eqnarray*}
with $a_j$ and $b_j$ are given respectively in \eqref{aj} and \eqref{b}. We see that $P$ contains a unique odd-degree term, i.e., $C_2s^{p-1}$, and $Q$ is even in $s$. Notice that $Q(0)=a_0\ne0$. Lemma \ref{odav} thus follows from differentiating \eqref{aqp} $2j+1$ times with respect to $s$, $0\le j\le\frac{p}{2}-2$, and then evaluating them at $s=0$ in turn.
\end{proof}

\begin{lemma}\label{doxi}
$\xi(\sigma)$ is even at $\sigma=0$ up to order $p-2$, i.e., $\frac{d^{2j+1}\xi}{d\sigma^{2j+1}}(0)=0$, $j=0,\cdots,\frac{p}{2}-2$.
\end{lemma}
\begin{proof}
Differentiating \eqref{cp} $2j+1$ times with respect to $\sigma$, $0\le j\le\frac{p}{2}-2$, leads to
\begin{align}
\label{dxi}\sigma\frac{d^{2j+2}\xi}{d\sigma^{2j+2}}=\frac{d^{2j+1}\xi}{d\sigma^{2j+1}}(\kappa_1\sqrt{\alpha}-2j-2)
+\sum_{i=0}^{2j}
  C_{2j+1}^i\frac{d^i\xi}{d\sigma^i}\frac{d^{2j+1-i}}{d\sigma^{2j+1-i}}
  (\kappa_1\sqrt{\alpha}).
\end{align}
Notice that $\frac{d}{d\sigma}=\frac{ds}{d\sigma}\frac{d}{ds}=\kappa_1\sqrt{\alpha}\xi\frac{d}{ds}$, $\kappa_1\sqrt{\alpha(0)}=1$, and either $i$ or $2j+1-i$ is odd. Lemma \ref{doxi} thus follows from Lemma \ref{odav} by induction on $j$.
\end{proof}

\begin{proposition}\label{dp}
$\frac{d^{p-1-2j}}{d\sigma^{p-1-2j}}(\xi(\sigma)^{2j-p+1})(0)=0$, $1\le j\le\frac{m}{2}-1$.
\end{proposition}
\begin{proof}
We have
\begin{eqnarray}
\notag&&\frac{d^{p-1-2j}}{d\sigma^{p-1-2j}}(\xi^{2j-p+1})\\
\label{der}&=&(2j-p+1)\xi^{2j-p}\frac{d^{p-1-2j}\xi}{d\sigma^{p-1-2j}}
+(\mbox{terms involving }\frac{d^i\xi}{d\sigma^i}\mbox{ with }i<p-1-2j).
\end{eqnarray}
It is easy to see that every term in the parenthesis on the right-hand side contains at least one odd-order derivative of $\xi$ as $p-1-2j$ is odd. Since $p-1-2j\le p-3$ for $1\le j\le\frac{m}{2}-1$, Proposition \ref{dp} thus follows from Lemma \ref{doxi}.
\end{proof}
Proposition \ref{dp} tells us that $\xi(\delta)^{2j-p+1}$ contains no $\delta^{p-1-2j}$ term for $1\le j\le\frac{m}{2}-1$, i.e., $s(\delta)^{2j-p+1}$ does not contribute to the constant term in the asymptotic expansion \eqref{rv} of $\mbox{Vol}_g(\{\sigma>\delta\})$. It remains to check the case $j=0$, for which the derivative \eqref{der} evaluated at $\sigma=0$ gives
\begin{align*}
\frac{d^{p-1}}{d\sigma^{p-1}}(\xi(\sigma)^{1-p})(0)=(1-p)\xi(0)^{-p}\frac{d^{p-1}\xi}{d\sigma^{p-1}}(0)
=(1-p)\frac{d^{p-1}\xi}{d\sigma^{p-1}}(0).
\end{align*}
We now go back to \eqref{dxi}. When $j=\frac{p}{2}-1$, it reads
\begin{align*}
\sigma\frac{d^p\xi}{d\sigma^p}=\frac{d^{p-1}\xi}{d\sigma^{p-1}}(\kappa_1\sqrt{\alpha}-p)+\xi\frac{d^{p-1}}
{d\sigma^{p-1}}(\kappa_1\sqrt{\alpha})+\sum_{i=1}^{p-2}C_{p-1}^i\frac{d^i\xi}{d\sigma^i}\frac{d^{p-1-i}}{d\sigma^{p-1-i}}
(\kappa_1\sqrt{\alpha}).
\end{align*}
With Lemmas \ref{odav} and \ref{doxi} in mind, we evaluate this equation at $\sigma=0$ to get
\begin{eqnarray*}
(1-p)\frac{d^{p-1}\xi}{d\sigma^{p-1}}(0)&=&-\frac{d^{p-1}}{d\sigma^{p-1}}(\kappa_1\sqrt{\alpha})(0)\\
&=&-\frac{\kappa_1^2}{2}\xi(0)(\kappa_1\sqrt{\alpha(0)}\xi(0))^{p-2}\frac{d^{p-1}\alpha}{ds^{p-1}}(0)\\
&=&-\frac{\kappa_1^2}{2}\frac{d^{p-1}\alpha}{ds^{p-1}}(0).
\end{eqnarray*}
Therefore
\begin{align*}
\frac{d^{p-1}}{d\sigma^{p-1}}(\xi(\sigma)^{1-p})(0)=-\frac{\kappa_1^2}{2}\frac{d^{p-1}\alpha}{ds^{p-1}}(0).
\end{align*}

In order to determine the value of $\frac{d^{p-1}\alpha}{ds^{p-1}}(0)$, we turn to \eqref{aqp}. Differentiating both sides $p-1$ times with respect to $s$, and then evaluating the resulting equation at $s=0$ gives
\begin{align*}
\frac{d^{p-1}\alpha}{ds^{p-1}}(0)Q(0)=C_2(p-1)!,
\end{align*}
i.e.,
\begin{align*}
\frac{d^{p-1}\alpha}{ds^{p-1}}(0)=(p-1)!\frac{C_2}{a_0}.
\end{align*}
We thus have
\begin{align*}
\frac{1}{(p-1)!}\frac{d^{p-1}}{d\sigma^{p-1}}(\xi(\sigma)^{1-p})(0)
=-\frac{\kappa_1^2C_2}{2a_0}.
\end{align*}

In summary, the renormalized volume is given by (cf. \eqref{asy})
\begin{eqnarray}
\notag V(g)\notag&=&\frac{C_1C_3}{\kappa_1^p}(\sum_{j=0}^{m/2-1}\frac{a_j}{2j-p+1}s_\ast^{2j-p+1}
  +\frac{\kappa_1^2C_2}{2(1-p)})\\
\notag&=&\frac{C_1C_3}{\kappa_1^p}(\sum_{j=0}^{m/2-1}\frac{a_j}{2j-p+1}s_\ast^{2j-p+1}
  +\frac{1}{2\nu}\sum_{j=0}^{m/2}\frac{b_j}{2j-p+1}s_\ast^{2j-p+1})\\
\notag&=&\frac{C_1C_3}{2\nu\kappa_1^ps_\ast^{p-1}}(\frac{2\nu a_0+b_0}{1-p}+\sum_{j=1}^{m/2-1}\frac{2\nu a_j+b_j}{2j-p+1}s_\ast^{2j}
  +\frac{b_{m/2}}{1-n}s_\ast^m)\\
\label{ff}&=&\frac{C_1C_3}{2\nu\kappa_1^ps_\ast^{p-1}}(\frac{\nu a_0}{1-p}+\sum_{j=1}^{m/2-1}\frac{\epsilon a_{j-1}+\nu a_j}{1-p+2j}s_\ast^{2j}
  +\frac{\epsilon}{1-n}s_\ast^m).
\end{eqnarray}
\begin{theorem}\label{rvv}
The formula \eqref{ff} gives the renormalized volumes associated to the even-dimensional conformally compact Einstein manifolds constructed in Theorems \ref{positive}, \ref{zero} and \ref{np}, where $C_1=\prod_{i=1}^rA_i^{n_i}$, $C_3=2\pi\mbox{Vol}_{g_N}(N)\prod_{i=1}^r\mbox{Vol}_{h_i}(V_i)$, $m=2+\sum_{i=1}^r2n_i$, $p=m+n$, $\kappa_1=\sqrt{\frac{1-p}{\nu}}$, $s_\ast=-\frac{1}{2A_1}$; $a_j$, $0\le j\le\frac{m}{2}-1$, are given by \eqref{aj}; $\nu$ and $A_i$, $1\le i\le r$, are as in Propositions \ref{fes} and \ref{fesnp}.
\end{theorem}
\begin{remark}
In Example \ref{hyp}, we recovered $4$-dimensional hyperbolic space $H^{4}$ whenever $\epsilon=-\frac{3}{\nu}$ and $A_1=\frac{3}{2\nu}$. Now we shall use \eqref{ff} to compute its renormalized volume. After a tedious but straightforward computation, we end up with $V(H^4)=\frac{4}{3}\pi^2$ (cf. Proposition \ref{4h}).
\end{remark}

\bibliographystyle{amsplain}
\bibliography{DCbib}

\end{document}